\documentclass{article}

\usepackage{amsthm,amsmath,amssymb}
\usepackage{mathrsfs}
\usepackage[dvipsnames]{xcolor}
\usepackage[colorlinks=true,
            linktocpage=true,
            citecolor=RubineRed,
            linkcolor=RoyalBlue]{hyperref}
\usepackage[numbers]{natbib}
\usepackage{doi}
\usepackage[a4paper,margin=2.5cm]{geometry}

  \newtheorem{theorem}{Theorem}[section]
  \newtheorem{proposition}[theorem]{Proposition}
  \newtheorem{lemma}[theorem]{Lemma}

\theoremstyle{definition}
  \newtheorem{definition}[theorem]{Definition}

  \newtheorem{example}[theorem]{Example}

\theoremstyle{remark}
  \newtheorem{remark}[theorem]{Remark}

\usepackage{stmaryrd}
\usepackage{tikz}
\usepackage{tikz-cd}
\usepackage{booktabs}
\usepackage{multicol}

\newcommand{\cp}{\trianglelefteqslant}
\newcommand{\ua}{{\uparrow}}  

\usepackage{color}

\DeclareSymbolFont{symbolsC}{U}{txsyc}{m}{n}
\DeclareMathSymbol{\cto}{\mathrel}{symbolsC}{128}
\DeclareMathSymbol{\dto}{\mathrel}{symbolsC}{130}
\DeclareMathSymbol{\diamondcto}{\mathrel}{symbolsC}{132}

\newcommand{\ftop}{1}
\newcommand{\fleq}{\preccurlyeq}
\newcommand{\fgeq}{\succcurlyeq}
\newcommand{\fmeet}{\curlywedge}

\DeclareMathOperator{\Fclp}{\mathcal{F}_{clp}}
\DeclareMathOperator{\Fcls}{\mathcal{F}_{closed}}
\DeclareMathOperator{\Ftop}{\mathcal{F}_{top}}
\newcommand{\topo}[1]{\mathbb{#1}}

\newcommand{\cat}[1]{\mathsf{#1}}

\newcommand{\cdash}[1][]{\mathrel{\kern.1ex\text{%
  \tikz[baseline=-.81ex, line width=.1ex, line cap=round, scale=1.1]
    {\draw (0ex,-.75ex) -- (0ex,.75ex);
     \draw (0ex,0ex) -- (1ex,0ex);
     \draw (-.4ex,0ex) arc(180:235:.91ex);
     \draw (-.4ex,0ex) arc(180:125:.91ex);}$_{#1}$}}\kern.1ex}

\newcommand{\genFil}{\mathbin{\triangledown}} 
\newcommand{\bigGenFil}{\operatorname{\bigtriangledown}}
\newcommand{\lb}{\llbracket} 
\newcommand{\rb}{\rrbracket} 

\newcommand{\mb}{\mathbb}

\usepackage[cal=dutchcal,scr=boondoxupr,scrscaled=1.050]{mathalfa}
\renewcommand{\log}[1]{\mathsf{#1}}   
\newcommand{\mo}[1]{\mathfrak{#1}}    
\newcommand{\lan}[1]{\mathbf{#1}}     
\newcommand{\amo}[1]{\mathscr{#1}}     
\renewcommand{\iff}{\quad\text{iff}\quad}
\newcommand{\llb}{\llbracket}
\newcommand{\rrb}{\rrbracket}
\newcommand{\llp}{\llparenthesis}     %
\newcommand{\rrp}{\rrparenthesis}
\newcommand{\Ax}{\mathbin{\mathrm{Ax}}}
\newcommand{\axref}[1]{\text{\ref{#1}}}
\newcommand{\Filt}{\mathcal{F}}
\renewcommand{\phi}{\varphi}
\newcommand{\Prop}{\mathrm{Prop}}

\usepackage{wasysym}
\let\oldDiamond\Diamond
\renewcommand{\Diamond}{%
  \mathchoice{\raisebox{-1pt}{$\displaystyle\oldDiamond$}}
             {\raisebox{-1pt}{$\oldDiamond$}}
             {\raisebox{-0.5pt}{$\scriptstyle\oldDiamond$}}
             {\raisebox{-0.2pt}{$\scriptscriptstyle\oldDiamond$}}}

\title{Sub-sub-intuitionistic logic}
\author{Jonte Deakin \& Jim de Groot \\
    \footnotesize{\texttt{u7480977@anu.edu.au jim@jimdegroot.com}} \\
    \footnotesize{The Australian National University, Canberra, Ngunnawal \& Ngambri Country, Australia}}
\date{}
\date{}

\begin{document}

\maketitle

\begin{abstract}
  Sub-sub-intuitionistic logic is obtained from intuitionistic logic by
  weakening the implication and removing distributivity.
  It can alternatively be viewed as conditional weak positive logic.
  We provide semantics for sub-sub-intuitionistic logic by means of
  semilattices with a selection function, prove a categorical duality
  for the algebraic semantics of the logic, and use this to derive completeness.
  We then consider the extension of sub-sub-intuitionistic logic
  with a variety of axioms.
\end{abstract}

%
%


\section{Introduction}

  Subintuitionistic logics are propositional logics that weaken the
  laws of intuitionistic implication.
  Since the 1980s they have been studied widely. Semantically, they can be
  obtained for example by replacing
  intuitionistic Kripke frames with Kripke frames that are only
  transitive~\cite{Vis81,Vis81t}, have an omniscient element~\cite{Res94},
  or satisfy no additional constraints at all~\cite{Cor87}.
  To find even weaker logics, interpretations of implication
  in (monotone) neighbourhood frames have been proposed in
  e.g.~\cite{JonMal18,JonMal19,GroPat22-pmwi}.
  Logically, all of these correspond to weakening the axioms of
  intuitionistic logic in various ways~\cite{Saz99,SuzOno97,Wan97}.

  In this paper we weaken intuitionistic logic in an orthogonal direction,
  namely by removing distributivity.
  It turns out that this forces us to move to a
  sub-intuitionistic setting, because defining implication
  as residuated with respect to conjunction forces distributivity.
  Therefore we obtain ``sub-sub-intuitionistic logic.''
  Since there is no canonical notion of implication,
  we define a minimal logical system of sub-sub-intuitionistic logic
  which can then be extended with additional axioms to obtain
  an implication of the desired strength.
  
  The fact that there is no single notion of implication comes as no surprise:
  it is a phenomenon observed and studied in the
  quantum logic literature~\cite{Har74,DalGiu08,YouSch19}.
  Besides, both non-distributivity~\cite{Cha19} as well as
  non-classical notions of implication~\cite{Lew73}
  are phenomena that play a role in the study of linguistics.
  With our work, we hope to provide a framework in which abstract
  non-distributive logics with various types of implications can be
  investigated and compared.
  
  We approach sub-sub-intuitionistic logic as a modal extension of
  not-necessarily-distributive positive logic, which we call \emph{weak positive logic}.
  Weak positive logic was first studied by Dalla Chiara in 1976~\cite{DalChi76},
  in an effort to define a general framework for quantum logics.
  Subsequently, non-distributive logics have been studied extensively, for example
  from the perspective of proof theory~\cite{DalChi76,ResPao05,CocSee01}
  and semantics~\cite{Urq78,ResPao05,Geh06,ConEA21,Gro24-vb-ll}.
  Also, there are various frame semantics and dualities for weak positive logics,
  using doubly ordered sets~\cite{Urq78}, polarities~\cite{Har92,HarDun97}
  and semilattices~\cite{BezEA24}.
  Besides in quantum logic, non-distributivity also appears in
  theoretical computer science \cite{Gir87,Moo97}
  and linguistics~\cite{Lam58,Vri17}.
  
  In the semilattice semantics from~\cite{BezEA24},
  formulas are interpreted as filters of a semilattice.
  While the intersection of two filters is again a filter,
  the union of two filters need not be. This leads to a non-standard
  interpretation of disjunctions, which prevents distributivity.
  A duality using semilattice semantics can be obtained as a restriction
  of the duality for semilattices, due to Hofmann, Mislove and
  Stralka~\cite{HofMisStr74}.

  We choose this semantics as the starting point for our semantic development
  of sub-sub-intuitionistic logic.
  Now there are several ways of adapting these with extra structure
  to interpret implication.
  For example, we can use a binary relation as in subintuitionistic logic~\cite{Cor87},
  or a ternary relation mimicking relevance logic~\cite{RouMey73,DunRes02}.
  Alternatively, we can equip our frames with a (monotone) neighbourhood
  structure, as used in~\cite{JonMal19,GroPat22-pmwi}.
  But we use yet another way, inspired by conditional logic,
  because it allows for a binary implication-like operator that
  satisfies almost no axioms.
  
  Classical conditional logic is the extension of classical propositional
  logic with a binary operator that preserves finite conjunctions in its
  second argument. It was introduced by Chellas in 1975~\cite{Che75}
  and has since been studied in depth, see for
  example~\cite{Che80,Gab85,Nut88,Seg89,UntSch14,BalCin18,Gir19,GirNegSba19}.
  The semantic mechanism used in conditional logic for interpreting
  implication is a \emph{selection function}.
  This is a function $s$ that assigns to each world $x$ and proposition $a$
  a proposition $s(x, a)$ consisting of worlds relevant to $x$ in the
  context of $a$. So a selection function frame is a pair $(X, s)$
  consisting of a non-empty set $X$ and a function
  $s : X \times \mathcal{P}X \to \mathcal{P}X$.
  (Alternatively, this can be viewed as a proposition-indexed
  collection of relations.) A formula of the form $\phi \cto \psi$ is then
  defined to be true
  at a world $x$ if all $y \in s(x, \llb \phi \rrb)$ satisfy $\psi$,
  where $\llb \phi \rrb$ denotes the set of worlds where $\phi$ holds.

  We extend the semilattice semantics from~\cite{BezEA24} with a selection
  function to obtain frame semantics for sub-sub-intuitionistic logic.
  Therefore, sub-sub-intuitionistic logic can be viewed as
  ``conditional weak positive logic,'' and following the conditional logic
  literature we shall denote the implication by $\cto$.
  Since formulas are now interpreted as filters of a semilattice $X$,
  the selection function has type $s : X \times \Filt X \to \Filt X$,
  where $\Filt X$ denotes the set of filters of $X$.
  As in the intuitionistic case~\cite{Wei19a}, 
  we impose coherence conditions between
  $s$ and the semilattice structure to ensure that all formulas are
  interpreted as filters. We call the resulting frames
  \emph{selection L-frames}, extending the terminology from~\cite{BezEA24}.
  
  We prove that the logic is sound with respect to the class of selection
  L-frames, and we extend them to a notion of \emph{general frames} for the logic.
  Interestingly, general frames do not only come equipped with a set of
  \emph{admissible filters},
  but also with a restriction on the selection function:
  its second argument acts only on admissible filters.
  This mirrors the literature on classical conditional logic~\cite{Seg89}.
  
  Before proving completeness, we derive a duality between the algebraic
  semantics of the logic on the one hand, and a topologised version of
  our frames called \emph{selection L-spaces} on the other hand.
  This builds on the duality between lattices
  and so-called L-spaces from~\cite{BezEA24}.
  As with general frames, the selection function only acts on certain filters,
  namely clopen filters.
  
  This leaves us with an interesting phenomenon: when proving completeness
  in the usual way we need to turn a selection L-space into a selection
  L-frame. However, the selection function of a selection L-space is only
  defined for clopen filters. So in order to obtain a frame
  we need to extend it to act on all filters.
  We call such an extension a \emph{fill-in} of the selection function.
  Many different fill-ins exist, and we use one of them to finish
  the completeness argument.

  Finally, we consider a number of axioms with which we can extend the basic
  logic and give sound and completeness semantics for the resulting logics.
  Interestingly, these require various notions of fill-in
  to obtain completeness.
  We use the extension results
  to elucidate the relation between sub-sub-intuitionistic
  logic and intuitionistic conditional logic,
  and to provide an alternative semantics for intuitionistic logic
  that is based on semilattices instead of preordered or partially ordered sets.

\subsection*{Outline of the paper}
  After recalling semilattice semantics for weak positive logic
  in Section~\ref{sec:prelim},
  we introduce sub-sub-intuitionistic logic in Section~\ref{sec:imp},
  and study its algebraic and relational semantics.
  In preparation of the duality,
  we also introduce general frames and morphisms between them in Section~\ref{sec:imp}.
  In Section~\ref{sec:duality} we prove a categorical duality between
  the algebraic semantics of the logic and a topologised version of
  the frame semantics. We use this to prove completeness of the basic system.
  We then proceed to investigate a variety of extensions of the logic
  with additional axioms, give correspondence results and sound and
  complete semantics for each of these in Section~\ref{sec:extensions}.
  We conclude by pointing out several avenues for further research 
  in Section~\ref{sec:conc}.

\section{A primer on weak positive logic}\label{sec:prelim}

  We briefly recall the semantics and duality of weak positive logic
  given in~\cite{BezEA24}. Throughout this paper we let $\Prop$ be some
  arbitrary but fixed set of proposition letters.
  Let $\lan{L}(\Prop)$ denote the language generated by the grammar
  \begin{equation*}
    \phi ::= p \mid \top \mid \bot \mid \phi \wedge \phi \mid \phi \vee \phi.
  \end{equation*}
  If no confusion arises we omit reference to $\Prop$ and simply
  write $\lan{L}$.
  Logics based on $\lan{L}$ are defined as as collections of
  \emph{consequence pairs}, which are expressions of the form $\phi \cp \psi$
  where $\phi, \psi \in \mathbf{L}$, similar to positive modal logic~\cite{Dun95}.
  
\begin{definition}\label{def:logic}
  Let $\log{L}$ be the smallest set of consequence pairs
  that is closed under uniform substitution as well as
  the following axioms and rules:
  \begin{align*}
    p \cp \top,
      &\qquad \bot \cp p,
      &\text{\emph{top} and \emph{bottom}} \\
    p \cp p,
      &\qquad \frac{p \cp q \quad q \cp r}{p \cp r},
      &\text{\emph{reflexivity} and \emph{transitivity}} \\
    p \wedge q \cp p,
       \qquad p \wedge q \cp q,
      &\qquad \dfrac{r \cp p \quad r \cp q}{r \cp p \wedge q},
      &\text{\emph{conjunction rules}} \\
    p \cp p \vee q,
       \qquad q \cp p \vee q,
      &\qquad \dfrac{p \cp r \quad q \cp r}{p \vee q \cp r}
      &\text{\emph{disjunction rules}}
  \end{align*}
  If $\Gamma$ is a set of consequence pairs then we let
  $\log{L}(\Gamma)$ denote the smallest set of consequence pairs
  closed under uniform substitution, the axioms and rules mentioned above, and those in $\Gamma$.
  We write $\phi \vdash_{\Gamma} \psi$ if $\phi \cp \psi \in \log{L}(\Gamma)$
  and $\phi \dashv\vdash_{\Gamma} \psi$ if $\phi \vdash_{\Gamma} \psi$
  and $\psi \vdash_{\Gamma} \phi$.
  If $\Gamma$ is the empty set then we abbreviate this to
  $\phi \vdash \psi$ and $\phi \dashv\vdash \psi$.
\end{definition}

  The algebraic semantics of $\log{L}(\Gamma)$ is given by
  lattices. We write $\cat{Lat}$ for the category of lattices
  and homomorphisms.

\begin{definition}
  Let $A$ be a lattice with operations $\top_A, \bot_A, \wedge_A, \vee_A$,
  and induced order $\leq_A$.
  A \emph{lattice model} is a pair $\amo{A} = (A, \sigma)$ 
  consisting of a lattice $A$ and an assignment $\sigma : \Prop \to A$
  of the proposition letters.
  The assignment $\sigma$ uniquely extends to a map
  $\llp \cdot \rrp_{\amo{A}} : \lan{L} \to A$
  by interpreting connectives via their lattice counterparts.
  
  We say that a lattice $A$ \emph{validates} a consequence pair $\phi \cp \psi$
  if $\llp \phi \rrp_{\amo{A}} \leq_A \llp \psi \rrp_{\amo{A}}$ for
  all lattice models $\amo{A}$ based on $A$, and denote this by $A \cdash \phi \cp \psi$.
  Given a set of consequence pairs $\Gamma$, we write
  $\cat{Lat}(\Gamma)$ for the full subcategory of $\cat{Lat}$ whose objects
  validate all consequence pairs in $\Gamma$.
\end{definition}
  
  As expected, this gives sound and complete algebraic
  semantics~\cite[Theorem~3.5]{BezEA24}:

\begin{theorem}\label{thm:alg-sem}
  We have $\phi \vdash_{\Gamma} \psi$ if and only if $\phi \cdash[\Gamma] \psi$.
\end{theorem}

  Formulas from $\lan{L}$ can be interpreted in
  meet-semilattices with a valuation~\cite{Dmi21}.
  The intuition behind this is that the collection of filters on a
  meet-semilattice is closed under arbitrary intersections.
  Therefore it forms a complete lattice, but disjunctions are not
  given by unions. This gives rise to a non-standard interpretation of
  disjunctions which prevents distributivity.

  In this paper, by a \emph{meet-semilattice} we mean a partially ordered
  set in which every finite subset has a greatest lower bound, called its \emph{meet}.
  The meet of $x$ and $y$ is denoted by $x \fmeet y$, reserving
  the symbol $\wedge$ for conjunctions of formulas.
  The empty meet is the top element, denoted by $1$.
  If $(X, 1, \fmeet)$ is a meet-semilattice then we write $\fleq$ for the
  partial order given by $x \fleq y$ iff $x \fmeet y = x$.
  A \emph{filter} of a meet-semilattice $(X, 1, \fmeet)$ is a subset $F$ of $X$
  which is upward closed under $\fleq$ and closed under finite meets.
  Filters are nonempty because they contain the empty meet, $1$.

\begin{definition}\label{def:L-model}
  An \emph{L-frame} is a meet-semilattice $(X, 1, \fmeet)$.
  An \emph{L-model} is an L-frame $(X, 1, \fmeet)$ together with a valuation
  $V$ that assigns to each proposition letter $p$ a
  filter $V(p)$ of $(X, 1, \fmeet)$.
  The interpretation of a formulas from $\mathbf{L}$ at a world $x$ of
  an L-model $\mo{M} = (X, 1, \fmeet, V)$ is defined recursively via
  \begin{align*}
    \mo{M}, x \Vdash p &\iff x \in V(p) \\
    \mo{M}, x \Vdash \top &\phantom{\iff}\text{always} \\
    \mo{M}, x \Vdash \bot &\iff x = 1 \\
    \mo{M}, x \Vdash \phi \wedge \psi &\iff \mo{M}, x \Vdash \phi
                     \text{ and } \mo{M}, x \Vdash \psi \\
    \mo{M}, x \Vdash \phi \vee \psi
      &\iff \exists y, z \in W \text{ s.t. } y \fmeet z \fleq x
            \text{ and }
            \mo{M}, y \Vdash \phi \text{ and } \mo{M}, z \Vdash \psi
  \end{align*}
  We denote the \emph{truth set} of $\phi$ by
  $\llb \phi \rrb^{\mo{M}} := \{ x \in X \mid \mo{M}, x \Vdash \phi \}$.
\end{definition}

  It can be shown that the truth set of every formula in any L-model $\mo{M}$
  is a filter in the underlying L-frame.
  In particular, the truth set of a formula of the form $\phi \vee \psi$
  is the smallest filter containing both $\llb \phi \rrb^{\mo{M}}$
  and $\llb \psi \rrb^{\mo{M}}$. We can explicitly write this filter as
  $
      \{ x \in X \mid y \fmeet z \fleq x \text{ for some }
           y \in \llb \phi \rrb^{\mo{M}} \text{ and }
           z \in \llb \psi \rrb^{\mo{M}} \}.
  $

\begin{definition}
  The collection of filters on an L-frame $\mo{X} = (X,\ftop,\fmeet)$ forms a lattice
  $(\mathcal{F}(\mo{X}),X,\{\ftop\},\cap,\genFil)$ where 
  \begin{equation*}
    p \genFil q := \ua\{x \fmeet y \in X \mid x \in p, y \in q\}.
  \end{equation*}
  We call this lattice the \emph{complex algebra} of $\mo{X}$ and denote
  it by $\mo{X}^+$.
\end{definition}

  In fact, the collection of filters on an L-frame forms a complete lattice,
  with arbitrary meets given nby intersections, and arbitrary joins given as follows:
  $$
  \bigGenFil Q := \bigcup \{q_1 \genFil \cdots \genFil q_n \mid n \in \mb{N} \text{ and } q_1,
  \ldots, q_n \in Q\}.
  $$

\begin{definition}\label{def:lmorphisms}
  An \emph{L-morphism} between L-frames
  $(X, 1, \fmeet)$ and $(X', 1', \fmeet')$ is a semilattice homomorphism
  $f : X \to X'$ such that for all $x \in X$ and $y', z' \in X'$: 
  \begin{itemize}
    \item $f(x) = 1'$ iff $x = 1$;
    \item If $y' \fmeet z' \fleq' f(x)$, then there exist $y, z \in X$ such that
          $y' \fleq' f(y)$ and $z' \fleq' f(z)$ and $y \fmeet z \fleq x$.
  \end{itemize}
  The second condition can be depicted as follows:
  
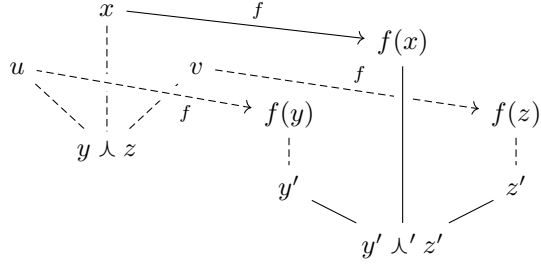
\begin{equation*}
    \begin{tikzcd}[scale=1]
        & [-1.3em]
          x \arrow[rrrd, "f"]
            \arrow[dddd, dashed, -]
        & [-1.3em]
        & [-1em]
        & [-1.5em]
        & [-1.5em] \\ [-2em]
        &
        &
        &
        & f(x)
        & \\ [-2.3em]
      u     \arrow[ddr, dashed, -]
        &
        & v \arrow[ddl, dashed, -]
            \arrow[rrrd, dashed, "f"]
        &&& \\ [-1.5em]
        &
        &
        & f(y) \arrow[dd, dashed, -]
               \arrow[lllu, dashed, <-, crossing over, "f" pos=.25]
        &
        & f(z) \arrow[dd, dashed, -] \\ [-2em]
        & y \fmeet z
        &&&& \\  [-2em]
        &
        &
        & y' \arrow[dr, -]
        &
        & z' \arrow[dl, -] \\ [-1.3em]
        &
        &
        &
        & y' \fmeet' z'  \arrow[uuuuu, crossing over, -]
        &
    \end{tikzcd}
  \end{equation*}
  An \emph{L-morphism} between L-models $(X, 1, \fmeet, V)$ and
  $(X', 1', \fmeet', V')$ is an L-morphism $f$ between the underlying
  frames such that $V(p) = f^{-1}(V'(p))$ for all $p \in \Prop$.
\end{definition}

  We obtain a duality for the category of lattices and lattice homomorphisms
  by using a topologised version of L-frames, called L-spaces.

\begin{definition} \label{def:lspaces}
  An \emph{L-space} is a tuple $\topo{X} = (X, \ftop, \fmeet,\tau)$ such that:
  \begin{enumerate}\itemsep=0pt
    \item $(X, \ftop, \fmeet)$ is a semilattice;
    \item $(X,\tau)$ is a compact topological space;
    \item If $a$ and $b$ are clopen filters of $\topo{X}$,
          then so is $a \genFil b$;
    \item $\topo{X}$ satisfies the \emph{HMS separation axiom:}
          \begin{equation*}
            \forall x, y \in X \; (\text{if } x \not \fleq y 
            \text{ then there exists a clopen filter } a \text{ s.t. } 
            x \in a \text{ and } y \notin a ).
          \end{equation*}
  \end{enumerate}
  An \emph{L-space morphism} is a continuous L-morphism. The category of
  L-spaces and L-space morphisms is denoted by $\cat{LSpace}$.
\end{definition}

  We write $\Fclp(\topo{X})$ the collection of clopen filters of an L-space $\topo{X}$.
  By construction, it forms a lattice with top, bottom, meet and join given
  by $X$, $\{ \ftop \}$, $\cap$ and $\genFil$.

\begin{definition}
  A \emph{clopen valuation} for an L-space $\topo{X}$ is a map
  $V : \Prop \to \Fclp(\topo{X})$.
  We call a pair $\topo{M} = (\topo{X}, V)$ of an L-space and a clopen valuation
  an \emph{L-space model}. The interpretation $\lb \phi\rb^\topo{M}$ of a formula
  $\phi$ in an L-space model $\topo{M} = (\topo{X}, \sigma)$ is defined as in
  the underlying L-model (see Definition~\ref{def:L-model}).

  An L-space model $\topo{M}$ \emph{validates} a consequence
  pair $\phi \cp \psi$ if $\lb \phi\rb^\topo{M} \subseteq \lb \psi \rb^\topo{M}$,
  and an L-space $\topo{X}$ \emph{validates} $\phi \cp \psi$ if
  $(\topo{X}, V)$ validates $\phi \cp \psi$ for every clopen valuation~$V$.
  We write $\topo{X} \Vdash \phi \cp \psi$ (or $\topo{M} \Vdash \phi \cp \psi$)
  if an L-space $\topo{X}$ (or model $\topo{M}$) validates $\phi \cp \psi$.
\end{definition}

  The assignment $\Fclp$ can be extended to a contravairant functor
  $\cat{LSpace} \to \cat{Lat}$ by defining the action of
  $\Fclp$ on an L-space morphism $f$ as $\Fclp(f) := f^{-1}$.
  In the converse direction, we have:

\begin{definition}
  Let $A$ be a lattice and write $(\Filt(A), A, \cap)$ for the semilattice of
  filters of $A$. Let $\tau_A$ be the topology on $\Filt(A)$ generated by
  \begin{equation*}
    \{\theta_A(a) \mid a \in A \} \cup \{ \theta_A(a)^c \mid a \in A \},
  \end{equation*}
  where $\theta_A(a) = \{ p \in \mathcal{F}(A) \mid a \in p\}$ and
  $\theta_A(a)^c = \mathcal{F}(A) \setminus \theta_A(a)$.
  Then the tuple $\Ftop(A) := (\Filt(A), A, \cap, \tau_A)$ is an L-space.
  Defining $\Ftop(h) := h^{-1}$ for a lattice homomorphism $h$
  yields the contravariant functor
  $\Ftop : \cat{Lat} \to \cat{LSpace}$.
\end{definition}

  We can now state the duality between lattices and L-spaces~\cite[Theorem 2.14]{BezEA24}.

\begin{theorem}\label{thrm:weakduality}
  The contravariant functors $\Fclp$ and $\Ftop$ establishes
  the categorical duality $\cat{LSpace} \equiv^{op} \cat{Lat}$.
\end{theorem}

  For future reference, we recall the units of the duality.
  These are given by
  $\theta : id_\cat{Lat} \to \Fclp\Ftop$ and $\eta : id_\cat{LSpace} \to \Ftop\Fclp$,
  defined on components via
  $\theta_A(a) = \{ p \in \Filt(A) \mid a \in p \}$ and
  $\eta_{\topo{X}}(x) = \{ a \in \Fclp(\mathbb{X}) \mid x \in a \}$.
  Besides, we note that for any lattice $A$, the clopen filters
  of $\Ftop(A)$ are exactly the filters of the form $\theta_A(a)$ for some $a \in A$.

\section{Weak positive logic with implication}\label{sec:imp}

  The goal of this paper is to extend weak positive logic with a notion of
  implication. Taking stock of existing logics such as intuitionistic logic,
  relevance logic and conditional logic, there appear to be several ways of
  doing so. Perhaps the most obvious method would be to
  define implication as residuated with respect to conjunction.
  Unfortunately, we are stopped dead in our tracks, because this would
  automatically make our logic distributive.

\begin{proposition}\label{prop:residdist}
  Suppose we extend $\log{L}$ from Definition~\ref{def:logic} with a binary
  operator $\to$ that satisfies the congruence rules as well as the following residuation
  rules
  \begin{equation*}
    \frac{\phi \land \psi \cp \chi}{\phi \cp \psi \cto \chi}
    \quad\text{and}\quad
    \frac{\phi \cp \psi \cto \chi}{\phi \land \psi \cp \chi}.
  \end{equation*}
  Then the resulting logic is distributive.
\end{proposition}
\begin{proof}
  The given rules would turn our algebraic semantics into Heyting algebras,
  which are known to be distributive~\cite[Proposition 1.5.3]{esakia}.
\end{proof}

  To keep our implication as flexible as possible, and maintain non-distributivity,
  we choose to mimic the approach of conditional logic. That is, we extend
  $\log{L}$ with a binary modal operator $\cto$ that is normal in its second
  argument, but does not satisfy any axioms for its first argument.

\subsection{Adding conditional implication}

  We extend the language from Section~\ref{sec:prelim} with a binary operator
  $\cto$ called \emph{conditional implication}. We use this to formally define
  sub-sub-intuitionistic logics and their algebraic semantics. 

\begin{definition}\label{def:syntax}
  Let $\Prop$ be some set of proposition letters and
  define $\mathbf{CL}(\Prop)$ to be the language generated by the grammar
  \begin{equation*}
    \phi ::= p \in \Prop \mid \top \mid \bot \mid \phi \land \phi \mid \phi \lor \phi \mid \phi \cto \phi.
  \end{equation*}
  We write $\mathbf{CL}$ when $\Prop$ is clear from context or irrelevant.
\end{definition}

  Just like Definition~\ref{def:logic}, sub-sub-intuitionistic logics
  are formulated as collections of consequence pairs. Henceforth, by
  ``consequence pair'' we mean an expression of the form $\phi \cp \psi$
  where $\phi$ and $\psi$ are taken from $\mathbf{CL}$.

\begin{definition}
  Let $\log{CL}$ be the smallest collection of consequence pairs
  that is closed under under uniform substitution, 
  the axioms and rules from Definition~\ref{def:logic}, and that also contains
  the following axioms:
  \begin{align*}
    \top &\cp p \cto \top
         &\text{\emph{modal top}} \\
    p \cto (q \land r) &\cp (p \cto q) \land (p \cto q)
    &\text{\emph{monotoncity}} \\
    (p \cto q) \land (p \cto r) &\cp p \cto (q \land r)
    &\text{\emph{normality}}
  \end{align*}
  and is closed under the congruence rules:
  \begin{equation*}
    \frac{p \cp q \quad q \cp p}
         {p \cto r \cp q \cto r}
    \quad\text{and}\quad
    \frac{p \cp q \quad q \cp p}
         {r \cto p \cp r \cto q}.
  \end{equation*}
  If $\Gamma$ is a set of consequence pairs then we let
  $\log{L}(\Gamma)$ denote the smallest set of consequence pairs that
  contains $\Gamma$ and is closed
  under uniform substitution and the axioms and rules mentioned above.
  We write $\phi \vdash_\Gamma \psi$ if $\phi \cp \psi \in \log{L}(\Gamma)$
  and $\phi \dashv\vdash_\Gamma \psi$ if both $\phi \vdash_\Gamma \psi$ and
  $\psi \vdash_\Gamma \phi$. 
  If $\Gamma = \emptyset$ then we write $\phi \vdash \psi$ and
  $\phi \dashv\vdash \psi$.
\end{definition}

  The algebraic semantics of the logic is given by lattices with a binary
  operator whose second argument preserves finite meets.

\begin{definition}\label{def:algebra}
  A \emph{conditional lattice} is a tuple
  $\amo{A} = (A, \top, \bot, \wedge, \vee, \cto)$
  such that $(A,\top,\bot,\land,\lor)$ is a lattice and 
  $\cto$ is a binary operator satisfying, for all $a, b, c \in A$:
  \begin{equation}\label{eq:cond-lattice}
    \top = a \cto \top
    \quad\text{and}\quad
    (a \cto b) \land (a \cto c) = a \cto (b \land c).
  \end{equation}
  The category of conditional lattices and homomorphisms is denoted by $\cat{CLat}$.
\end{definition}
  A \emph{conditional lattice model} is a pair
  $\amo{M} = (\amo{A}, \sigma)$ consisting of a conditional lattice $\amo{A}$
  and an assignment $\sigma : \Prop \to A$ of the proposition letters.
  The assignment $\sigma$ extends uniquely to a map
  $\llp \cdot \rrp_{\amo{M}} : \mathbf{CL} \to \amo{A}$ by interpreting
  connectives via their conditional lattice counterparts.
  A conditional lattice model $\amo{M}$ \emph{validates} a consequence pair
  $\phi \cp \psi$ if $\llp \phi \rrp_{\amo{M}} \leq \llp \psi \rrp_{\amo{M}}$,
  where $\leq$ is the order underlying the lattice.
  A conditional lattice $\amo{A}$ is said to \emph{validate} a consequence
  pair $\phi \cp \psi$ if every conditional lattice  model of the form
  $(\amo{A}, \sigma)$ validates it.

\begin{definition}
  Let $\Gamma$ be a set of consequence pairs.
  Then $\cat{CLat}(\Gamma)$ is defined as the full subcategory of $\cat{CLat}$ whose
  objects validate all consequence pairs in $\Gamma$.
  We write $\phi \cdash[\Gamma] \psi$ if every $\amo{A} \in \cat{CLat}(\Gamma)$
  validates $\phi \cp \psi$
  and abbreviate $\phi \cdash[\emptyset] \psi$ to $\phi \cdash \psi$.
\end{definition}

  We note that $\cat{CLat}(\Gamma)$ forms a variety of algebras,
  given by the equations defining lattices together with the equations
  obtained from $\Gamma$ by replacing $\cp$ with $\leq$ and viewing proposition
  letters as variables.
  A standard Lindenbaum argument, akin to that in~\cite[Section~3.1]{BezEA24},
  can be used to prove that conditional lattices indeed provide the algebraic semantics
  logics of the form $\log{CL}(\Gamma)$.

\begin{theorem}\label{thm:alg-sound-compl}
  Let $\Gamma \cup \{ \phi \cp \psi \}$ be a set of consequence pairs. Then
  \begin{equation*}
    \phi \vdash_{\Gamma} \psi \iff \phi \cdash[\Gamma] \psi.
  \end{equation*}
\end{theorem}

  We complete this subsection by verifying that the logic we have defined
  does not unexpectedly satisfy distributivity.

\begin{example}
  Let $(A, \top, \bot, \wedge, \vee)$ be any lattice. Then defining a binary
  operator $\cto'$ on $A$ by $a \cto' b = \top$ for all $a, b \in A$
  yields a conditional lattice $\amo{A} = (A, \top, \bot, \wedge, \vee, \cto')$.
  Since the lattice we start with need not be distributive, the same holds
  for $\amo{A}$.
  Therfore 
  $p \wedge (q \vee r) \cp (p \wedge q) \vee (p \wedge r)$ is not a theorem of $\log{CL}$.
\end{example}

  In a similar way, it can be shown that none of the axioms considered
  in Section~\ref{sec:extensions} can be derived in $\log{CL}$.

\subsection{Selection L-frames}\label{subsec:sel-frm}

  The standard way of interpreting conditional implication in a classical
  setting is via a selection function~\cite[Section~3]{Che75}.
  In this context, a selection function frame consists of a non-empty set $X$
  of worlds and a function $s : X \times \mathcal{P}X \to \mathcal{P}X$, where
  $\mathcal{P}X$ denotes the powerset of $X$.
  Alternatively the selection function can be viewed as a
  subset-indexed collection of relations
  $\{ R_a \mid a \subseteq X \}$~\cite{Seg89},
  or as a ternary relation $R \subseteq X \times X \times \mathcal{P}X$~\cite[Definition~3.1]{UntSch14}.

  This approach was adapted to an intuitionistic setting
  by Weiss~\cite{Wei19a}, who equipped an intuitionistic Kripke frame
  $(X, \leq)$ with a subset-indexed set of relations $\{ R_a \mid a \subseteq X \}$
  such that $({\leq} \circ R_a) \subseteq (R_a \circ {\leq})$ for all relations.
  However, by the nature of the interpretation of conditional implication,
  we only need relations $R_a$ where $a$ is upward closed in $(X, \leq)$.
  Taking this into account, we may refine Weiss' semantics to an
  intuitionistic Kripke frame $(X, \leq)$ with a selection function
  $s : X \times \mathcal{up}(X, \leq) \to \mathcal{up}(X, \leq)$
  such that $x \leq y$ implies $s(y, a) \subseteq s(x, a)$
  for all $a \in \mathcal{up}(X, \leq)$.
  (Here $\mathcal{up}(X, \leq)$ is the collection of upsets of $(X, \leq)$.)
  
  While classically any subset of a frame may serve as the
  interpretation of a formula, while intuitionistically
  only upward closed subsets are used as the interpretations of formulas. 
  Guided by this observation, we define frame semantics for sub-sub-intuitionistic
  logic by extending L-frames $(X, \ftop, \fmeet)$ with a selection function
  of the form
  \begin{equation*}
    s : X \times \Filt(X, \ftop, \fmeet) \to \Filt(X, \ftop, \fmeet).
  \end{equation*}
  In order to ensure that all formulas are interpreted as filters,
  we impose three coherence conditions between the semilattice
  structure of the frame and the selection function.
  Thus, we arrive at the following definition.

\begin{definition}
  \label{def:selectionfunctionlframe}
  Let $(X, \ftop, \fmeet)$ be an L-frame.
  A \emph{selection function} is a function
  $s : X \times \Filt(X, \ftop, \fmeet) \to \Filt(X, \ftop, \fmeet)$ such that
  for all $x, y, z \in X$ and $a \in \Filt(X, \ftop, \fmeet)$:
  \begin{enumerate} 
    \renewcommand{\labelenumi}{(\theenumi) }
    \renewcommand{\theenumi}{S$_{\arabic{enumi}}$}
    \item \label{it:sf-top}
          $s(\ftop, a) = \{ 1 \}$;
    \item \label{it:sf-up}
          If $x \fleq y$ then $s(y, a) \subseteq s(x, a)$;
    \item \label{it:sf-filt}
          If $z \in s(x \fmeet y, a)$ then there exist
          $u \in s(x, a)$ and $v \in s(y, a)$ such that $u \fmeet v \fleq z$.
  \end{enumerate}
  A \emph{selection L-frame} is an L-frame with a selection function.
\end{definition}

\begin{definition}\label{def:sf-model}
  A \emph{selection L-model} is a selection L-frame
  $(X, \ftop, \fmeet, s)$ with a valuation $V : \text{Prop} \to \mathcal{F}(X,\ftop,\fmeet)$.
  The interpretation of a $\mathbf{CL}$-formula at a world $x \in X$
  is defined recursively by the clauses from Definition~\ref{def:L-model}
  with the additional case:
  \begin{equation*}
    \mo{M}, x \Vdash \phi \cto \psi
      \iff s(x, \llb \phi \rrb^{\mo{M}}) \subseteq \llb \psi \rrb^{\mo{M}}.
  \end{equation*}
  Recall that $\llb \phi \rrb^{\mo{M}} = \{ x \in X \mid \mo{M}, x \Vdash \phi \}$
  denotes the \emph{truth set} of $\phi$ in $\mo{M}$.
  When $\mo{M}$ is clear from context we drop the superscript.

  We write $\mathfrak{M} \Vdash \phi \cp \psi$, and say that $\mo{M}$ \emph{validates}
  $\phi \cp \psi$, if $\llb \phi \rrb^{\mo{M}} \subseteq \llb \psi \rrb^{\mo{M}}$.
  Similarly, a selection L-frame $\mo{X}$ \emph{validates} $\phi \cp \psi$,
  denoted by $\mo{X} \Vdash \phi \cp \psi$, if every model of the form
  $(\mo{X}, V)$ validates $\phi \cp \psi$.
  Finally, if $\Gamma$ is a set of consequence pairs, then we write
  $\mo{X} \Vdash \Gamma$ if $\mo{X}$ validates every consequence pair in $\Gamma$,
  and $\phi \Vdash_{\Gamma} \psi$ if $\mo{X} \Vdash \Gamma$ implies
  $\mo{X} \Vdash \phi \cp \psi$ for all selection L-frames $\mo{X}$.
\end{definition}

  We verify that selection L-models do indeed satisfy the
  persistence condition. In other words, we show that the interpretation
  of any formula in a selection function L-model forms a filter.
  We make use of the following lemma, which can be proven using
  the three selection function conditions from Definition~\ref{def:sf-model}.

\begin{lemma}\label{lem:sf-genfil}
  Let $\mo{X} = (X, \ftop, \fmeet, s)$ be a selection L-frame.
  Then for every filter $a$ of $(X, \ftop, \fmeet)$ and all $x, y \in X$ we have
  $s(x \fmeet y, a) = s(x, a) \genFil s(y, a)$.
\end{lemma}

\begin{proposition}\label{prop:persistence}
  Let $\mathfrak{M} = (X, \ftop, \fmeet, s, v)$ be a selection L-model.
  Then $\llbracket \phi \rrbracket$ is a filter on $(X, \ftop, \fmeet)$
  for every $\phi \in \mathbf{CL}$.
\end{proposition}
\begin{proof}
  We proceed by induction on the complexity of $\phi$.
  The base cases and the inductive cases for conjunctions and disjunctions are as
  in~\cite[Lemma~3.8]{BezEA24}, so we focus on the case $\phi = \psi \cto \chi$.
  
  It follows from~\eqref{it:sf-top} and~\eqref{it:sf-up} that
  $\llb \psi \cto \chi \rrb$ contains $\ftop$ and is upwards closed.
  For closure under meets, suppose $x, y \in \llb \psi \cto \chi \rrb$.
  Then $s(x, \llb \psi \rrb) \subseteq \llb \chi \rrb$
  and $s(y, \llb \psi \rrb) \subseteq \llb \chi \rrb$.
  Lemma~\ref{lem:sf-genfil} and the fact that $\llb \chi \rrb$ is a
  filter then imply
  $s(x \fmeet y, \llb \psi \rrb) = s(x, \llb \psi \rrb) \genFil s(y, \llb \psi \rrb) \subseteq \llb \chi \rrb$.
  Therefore $x \fmeet y \in \llb \psi \to \chi \rrb$.
\end{proof}

  We close this subsection by showing soundness of $\log{CL}(\Gamma)$ with
  respect to selection L-frames validating $\Gamma$. Completeness will have to wait until
  we establish a categorical duality in Section~\ref{sec:duality}.
  
\begin{definition}
  The \emph{complex algebra} of a selection L-frame
  $\mo{X} = (X, \ftop, \fmeet, s)$ is the conditional lattice
  $\mo{X}^+ := (\Filt(X, \ftop, \fmeet), \dto)$,
  where $\Filt(X, \ftop, \fmeet)$ is the lattice of filters of $(X, \ftop, \fmeet)$
  and $\dto$ is given by
  \begin{equation*}
    p \dto q = \{ x \in X \mid s(x, p) \subseteq q\}.
  \end{equation*}
\end{definition}

\begin{lemma}
  For any selection L-frame $\mo{X} = (X, \ftop, \fmeet, s)$,
  its complex algebra $\mo{X}^+$ is a conditional lattice.
\end{lemma}
\begin{proof}
  We know that $\Filt(X, \ftop, \fmeet)$ forms a lattice (with top element $X$
  and conjunction $\cap$),
  so we only have to verify that $\dto$ satisfies the equations
  from~\eqref{eq:cond-lattice}. For the first equation, note that
  $a \dto X = \{ x \in X \mid s(x, a) \subseteq X \} = X$.
  For the second equation, let $a, b, c \in \Filt(X, \ftop, \fmeet)$ and compute
  \begin{align*}
    a \dto (b \cap c)
      &= \{ x \in X \mid s(x, a) \subseteq b \cap c \} \\
      &= \{ x \in X \mid s(x, a) \subseteq b \} \cap \{ x \in X \mid s(x, a) \subseteq c \}
       = (a \dto b) \cap (a \dto c).
  \end{align*}
  The proves the lemma.
\end{proof}

\begin{lemma}\label{lem:cpx-summary}
  Let $\mo{X}$ be a selection function L-frame, $V$ any valuation for $\mo{X}$
  and $\phi, \psi \in \mathbf{CL}$. Then:
  \begin{enumerate}
    \item $\llb \phi \rrb^{(\mo{X}, V)} = \llp \phi \rrp_{(\mo{X}^+, V)}$
    \item $(\mo{X}, V) \Vdash \phi \cp \psi$ iff $(\mo{X}^+, V) \cdash \phi \cp \psi$
    \item \label{it:cpx-validity}
          $\mo{X} \Vdash \phi \cp \psi$ iff $\mo{X}^+ \cdash \phi \cp \psi$
  \end{enumerate}
\end{lemma}
\begin{proof}
  The first item follows from a routine induction on the structure of $\phi$.
  The second item follows from the first and the fact that $\mo{X}^+$ is ordered
  by inclusion. The third item follows from the second and the fact that
  valuations for $\mo{X}$ correspond bijectively with assignments for $\mo{X}^+$.
\end{proof}

  Finally, we can use Lemma~\ref{lem:cpx-summary} to prove soundness
  of $\log{CL}(\Gamma)$ with respect to classes of selection function L-frames.

\begin{theorem}\label{thrm:framesound}
  Let $\Gamma \cup \{ \phi \cp \psi \}$ be a collection of consequence pairs.
  Then
  \begin{equation*}
    \phi \vdash_{\Gamma} \psi
    \quad\text{implies}\quad
    \phi \Vdash_{\Gamma} \psi.
  \end{equation*}
\end{theorem}
\begin{proof}
  It follows from Lemma~\ref{lem:cpx-summary}\eqref{it:cpx-validity}
  that $\phi \cdash[\Gamma] \psi$ implies $\phi \Vdash_{\Gamma} \psi$.
  Combining this with Theorem~\ref{thm:alg-sound-compl} yields the
  desired result.
\end{proof}

\subsection{General frames}

  In modal logic, general frames are obtained by equipping frames with a
  collection of ``admissible'' subsets. That it, it adds a collection of
  designated subsets which are allows to be interpretations of formulas.
  They are often used to bridge the gap between frame semantics and
  algebraic semantics of a logic, see for example~\cite[Section~1.4]{BlaRijVen01}.
  
  In our setting, general frames generalise selection L-frames from
  Definition~\ref{def:selectionfunctionlframe} as well as the topologised
  frames used in the next section.
  Therefore, they allow us to define a notion of morphism between frames
  once, and specialise it the the setting of selection L-frames
  as well as the spaces used for the duality.
  Additionally, the use of general frames lets us give frame correspondents
  of formulas for both selection L-frames and topologised frames at once,
  preventing unnecessary duplication.

\begin{definition}
  A \emph{general selection L-frame}, or \emph{general frame} for short,
  is a tuple $\mo{G} = (X, \ftop, \fmeet, s, A)$ consisting of
  \begin{itemize}
    \item an L-frame $(X, \ftop, \fmeet)$;
    \item a \emph{selection function} $s : X \times A \to \Filt(X, \ftop, \fmeet)$
          satisfying~\eqref{it:sf-top}, \eqref{it:sf-up} and~\eqref{it:sf-filt}
          for all $a \in A$;
    \item a collection $A \subseteq \Filt(X, \ftop, \fmeet)$ of
          \emph{admissible subsets} that contain $X$ and $\{ \ftop \}$,
          and is closed under $\cap, \genFil$ and $\dto$.
  \end{itemize}
  An \emph{admissible valuation} is a map $V : \Prop \to A$,
  and a general frame together with an admissible valuation is called
  a \emph{general selection L-model} or \emph{general model}.

  The interpretation of formulas in a general model are defined using the
  clauses from Definitions~\ref{def:L-model} and~\ref{def:sf-model}.
  Validity of a consequence pair $\phi \cp \psi$ in a general frame or model
  is defined as in Definition~\ref{def:sf-model}.
  If $\Gamma \cup \{ \phi \cp \psi \}$ is a collection of consequence pairs,
  then we write $\phi \Vdash^g_{\Gamma} \psi$ if
  every general frame validating $\Gamma$ also validates $\phi \cp \psi$.
\end{definition}

  Let $\mo{G} = (X, \ftop, \fmeet, s, A)$ be a general frame
  and $\mo{M} = (\mo{G}, V)$ a general model.
  Then the closure conditions on $A$ ensure that $\llb \phi \rrb^{\mo{M}} \in A$
  for all $\phi \in \mathbf{CL}$.
  In fact, the set $A$ forms a conditional lattice.

\begin{definition}\label{def:gfrm-cpx}
  The \emph{complex algebra} of a general frame $\mo{G} = (X, \ftop, \fmeet, s, A)$
  is given by $\mo{G}^+ := (A, X, \{ \ftop \}, \cap, \genFil, \dto)$.
\end{definition}

  Every selection L-frame $\mo{X} = (X, \ftop, \fmeet, s)$ can be
  viewed as a general frame by letting all filters be admissible, that is,
  by setting $A = \Filt(X, \ftop, \fmeet)$. Then any valuation for $\mo{X}$
  is admissible, and validity of consequence pairs for $\mo{X}$
  and its corresponding general frame coincide.
  General frames with such a trivial set of admissible subsets are sometimes
  called \emph{full} (for example, see \cite[Section~2]{Seg89}).
  When convenient, we view selection L-frames as full general
  frames.

  We can generalise Lemma~\ref{lem:cpx-summary} and
  Theorem~\ref{thrm:framesound} to the setting of general frames.
  
\begin{lemma}\label{lem:gen-cpx-summary}
  Let $\mo{G}$ be a general frame, $V$ an admissible valuation for $\mo{G}$
  and $\phi, \psi \in \mathbf{CL}$. Then:
  \begin{enumerate}
    \item $\llb \phi \rrb^{(\mo{G}, V)} = \llp \phi \rrp_{(\mo{G}^+, V)}$
    \item $(\mo{G}, V) \Vdash \phi \cp \psi$ iff $(\mo{G}^+, V) \cdash \phi \cp \psi$
    \item \label{it:gen-cpx-validity}
          $\mo{G} \Vdash \phi \cp \psi$ iff $\mo{G}^+ \cdash \phi \cp \psi$
  \end{enumerate}
\end{lemma}
  
\begin{theorem}\label{thm:gen-frm-sound}
  Let $\Gamma \cup \{ \phi \cp \psi \}$ be a collection of consequence pairs.
  Then
  \begin{equation*}
    \phi \vdash_{\Gamma} \psi
    \quad\text{implies}\quad
    \phi \Vdash^g_{\Gamma} \psi.
  \end{equation*}
\end{theorem}

\subsection{Selection morphisms}

  In order to prove a categorical duality for varieties of conditional lattices,
  we need to think about morphisms between frames. In the classical setting,
  morphisms between selection function frames were defined implicitly
  in~\cite{KupPat11}.
  Here, we define morphisms between
  general frames, so that their definition restricts to a notion of morphism
  between selection L-frames by taking all filters to be admissible,
  as well as the topologised frames used in Section~\ref{sec:duality}.

\begin{definition}\label{def:smor}
  Let $\mo{G} = (X, \ftop, \fmeet, s, A)$ and $\mo{G}' = (X', \ftop', \fmeet', s', A')$
  be two general frames.
  Then a \emph{selection morphism} from $\mo{G}$ to $\mo{G}'$
  is an L-morphism $f$ from $(X, \ftop, \fmeet)$ to $(X', \ftop', \fmeet')$
  such that, for all $x \in X$ and $a' \in A'$ the following hold:
  \begin{enumerate} 
    \renewcommand{\labelenumi}{(\theenumi) }
    \renewcommand{\theenumi}{M$_{\arabic{enumi}}$}
    \setcounter{enumi}{-1}
    \item \label{it:smor-0}
          $f^{-1}(a') \in A$;
    \item \label{it:smor-1}
          If $y \in s(x, f^{-1}(a'))$ then $f(y) \in s'(f(x), a)$;
    \item \label{it:smor-2}
          If $y' \in s'(f(x), a')$ then there exists an $y \in s(x, f^{-1}(a'))$
          such that $f(y) \leq y'$.
  \end{enumerate}
\end{definition}
  
  Conditions~\eqref{it:smor-1} and~\eqref{it:smor-2} can be combined to the
  equivalent condition
  \begin{equation}\label{eq:mor-alt}
    s'(f(x), a') = \ua f[s(x, f^{-1}(a'))].
  \end{equation}
  However, the separation into two conditions will simplify the proof
  that selection morphisms preserve and reflect truth of the conditional
  implication.

\begin{definition}
  Let $\mo{M} = (\mo{G}, V)$
  and $\mo{M}' = (\mo{G}', V')$
  be two general models.
  Then a \emph{selection model morphism} from $\mo{M}$ to $\mo{M}$
  is a selection morphism $f : \mo{G} \to \mo{G}'$
  such that $V = f^{-1} \circ V'$.
\end{definition}

  Selection morphisms specialise to selection L-frames and models
  by viewing every filter as admissible. Thus, they are given by L-morphisms
  that satisfy~\eqref{it:smor-1} and~\eqref{it:smor-2}
  where $a'$ ranges over all filters of the codomain.

  Every selection morphism between two (general) frames gives rise to a
  conditional lattice homomorphism between the corresponding complex algebras.
  
\begin{lemma}\label{lem:smor-dual}
  Let $\mo{G} = (X, \ftop, \fmeet, s, A)$ and
  $\mo{G}' = (X', \ftop', \fmeet', s', A')$ be two general frames
  and $f : \mo{G} \to \mo{G}'$ a selection morphism.
  Then $f^{-1} : (\mo{G}')^+ \to \mo{G}^+$ is a conditional lattice homomorphism.
\end{lemma}
\begin{proof}
  It follows from the definition of an L-morphism that $f^{-1}$ is a lattice
  homomorphism, so we only have to verify that
  $f^{-1}(a' \dto b') = f^{-1}(a') \dto f^{-1}(b')$ for all
  $a', b' \in A'$.
  Using~\eqref{eq:mor-alt} and the fact that $b'$ is upward closed
  we find
  \begin{alignat*}{2}
    x \in f^{-1}(a' \dto b')
      &\iff s'(f(x), a') \subseteq b'
      &&\iff f[s(x, f^{-1}(a'))] \subseteq b' \\
      &\iff s(x, f^{-1}(a')) \subseteq f^{-1}(b')
      &&\iff x \in f^{-1}(a') \dto f^{-1}(b'),
  \end{alignat*}
  which proves the desired equality.
\end{proof}

  We can combine Lemmas~\ref{lem:gen-cpx-summary} and~\ref{lem:smor-dual}
  to show that selection model morphisms preserve and reflect truth of
  formulas.

\begin{theorem}\label{thrm:prv}
  Let $\mo{G} = (X, \ftop, \fmeet, s, A)$ and $\mo{G}' = (X', \ftop', \fmeet', s', A')$
  be two general frames, $V$ and $V'$ two admissible valuations for them,
  and $\mo{M} = (\mo{G}, V)$ and $\mo{M}' = (\mo{G}', V')$ two general models.
  If $f : \mathfrak{M} \to \mathfrak{M}'$ is a selection model morphism,
  then for all $x \in X$ and $\phi \in \mathbf{CL}$ we have
  \begin{equation*}
    \mo{M}, x \Vdash \phi \iff \mo{M}', f(x) \Vdash \phi.
  \end{equation*}
\end{theorem}
\begin{proof}
  It follows from Lemma~\ref{lem:smor-dual} that
  $f^{-1} : (\mo{G}')^+ \to \mo{G}^+$ is a conditional lattice homomorphism.
  The fact that $V = f^{-1} \circ V'$ 
  then implies that $\llp \cdot \rrp_{(\mo{G}^+, V)} = f^{-1} \circ \llp \cdot \rrp_{((\mo{G}')^+, V')}$.
  Combining this with Lemma~\ref{lem:gen-cpx-summary} proves the theorem.
\end{proof}

\section{Duality and completeness}\label{sec:duality}

  In this section we derive a duality between the category of conditional
  lattices and a category of certain topologised selection L-frames.
  We build this on the duality for lattices from Theorem~\ref{thrm:weakduality}.
  Adding the conditional implication to lattices is dually reflected by
  equipping L-spaces with a selection function that acts on clopen filters.
  We will see that it can be viewed as a general frame, where the clopen
  filters are the admissible filters. This allows us to use the notion of
  a selection morphism from Definition~\ref{def:smor} to define a suitable
  notion of morphism between our spaces.
  In Section~\ref{sec:scf}, we use the duality to obtain
  completeness for the basic system $\log{CL}$ with respect to the class of
  all selection L-frames.

\subsection{Spaces and Duality}

  Our definition of topological semantics takes L-spaces from
  Definition~\ref{def:lspaces} and attaches a \emph{topological selection function}.
  Recall that we denote the lattice of clopen filters of an L-space
  $\topo{X}$ by $\Fclp(\topo{X})$.
  We denote the set of closed filters $\Fcls(\topo{X})$.

\begin{definition}\label{def:condlspace}
  Let $\mathbb{X} = (X, \ftop, \fmeet, \tau)$ be an L-space.
  A \emph{topological selection function} is a function
  $s : X \times  \Fclp(\topo{X}) \to \mathcal{F}_{\text{Closed}}(\topo{X})$
  such that for all $x, y, z \in X$ and $a, b \in \Fclp(\topo{X})$ we have:
  \begin{enumerate}
    \renewcommand{\labelenumi}{(\theenumi) }
    \renewcommand{\theenumi}{S$_{\arabic{enumi}}$}
    \item \label{it:S1d}
          $s(\ftop, a) = \{ 1 \}$;
    \item \label{it:S2d}
          If $x \fleq y$ then $s(y, a) \subseteq s(x, a)$;
    \item \label{it:S3d}
          If $z \in s(x \fmeet y, a)$ then there exist
          $u \in s(x, a)$ and $v \in s(y, a)$ such that $u \fmeet v \fleq z$.
    \renewcommand{\theenumi}{S$_t$}
    \item \label{it:St}
          If $a, b \in \Fclp(\topo{X})$ then
          $a \dto b := \{ x \in X \mid s(x, a) \subseteq b \} \in \Fclp(\topo{X})$.
  \end{enumerate}
  A \emph{selection L-space} is an L-space with a topological
  selection function.
\end{definition}

\begin{remark}
  Selection function L-spaces can be viewed as general frames in a canonical way.
  If $\topo{X} = (X, \ftop, \fmeet, s, \tau)$ is a selection L-space
  then the set $\Fclp(\topo{X})$ of clopen filters of $\topo{X}$ contains $X$
  and $\{ \ftop \}$ and is closed under
  $\cap$ and $\genFil$ because $\topo{X}$ is based on an L-space.
  Furthermore, $\Fclp(\topo{X})$ is closed under $\dto$ by definition.
  So $(X, \ftop, \fmeet, s, \Fclp(\topo{X}))$ is a general frame.
  We call general frames that arise in this way \emph{descriptive}.
  We will use this perspective of selection L-spaces as general
  frames to define morphisms (in Definition~\ref{def:sflsm} below),
  and to transport frame correspondence results from general frames to
  selection L-spaces in Section~\ref{sec:extensions}.
\end{remark}

A \emph{selection morphism} between two selection L-spaces
$\topo{X}$ and $\topo{X}'$ is a function $f$ that is a selection morphism
between the corresponding general frames. We work this out concretely:

\begin{definition}
  \label{def:sflsm}
  Let $\topo{X} = (X, \ftop, \fmeet, s, \tau)$ and 
  $\topo{X}' = (X', \ftop', \fmeet', s', \tau')$ be two selection L-spaces.
  A \emph{selection morphism} from $\topo{X}$ to $\topo{X}'$ is a continuous
  L-frame morphism $f$ such that for all $x \in X$ and all clopen filters
  $a'$ of $\topo{X}'$:
  \begin{enumerate}
    \renewcommand{\labelenumi}{(\theenumi) }
    \renewcommand{\theenumi}{M$_{\arabic{enumi}}$}
    \item \label{it:smor-1d}
          if $y \in s(x, f^{-1}(a'))$ then $f(y) \in s'(f(x), a)$;
    \item \label{it:smor-2d}
          if $y' \in s'(f(x), a')$ then there exists an $y \in s(x, f^{-1}(a'))$
          such that $f(y) \leq y'$.
  \end{enumerate}
  The collection of selection L-spaces and selection morphisms forms a
  category, which we denote by $\cat{SLSpace}$.
\end{definition}

\begin{definition}
  A \emph{selection L-space model} is a selection 
  L-space $\topo{X}$ with a valuation  $V : \Prop \to \Fclp(\topo{X})$.
  The interpretation of formulas in a selection function L-space model is 
  defined via the clauses from Definitions~\ref{def:L-model}
  and~\ref{def:sf-model}.
  A selection L-space model $\mathbb{M}$ \emph{validates} a consequence pair
  $\phi \cp \psi$ if $\lb\phi\rb^\mathbb{M}
  \subseteq \lb\psi\rb^\mathbb{M}$. A selection L-space $\mathbb{X}$
  validates a consequence pair $\phi \cp \psi$ if every selection function L-space model
  $\mathbb{M}$ of the form $(\mo{X}, V)$ validates $\phi \cp \psi$. We write $\mo{X} \Vdash \phi
  \cp \psi$ if $\mo{X}$ validates $\phi \cp \psi$.
\end{definition}

  We claim that the category $\cat{SLSpace}$ is dual to $\cat{CLat}$.
  To prove this, we extend the functors $\Fclp$ and $\Ftop$ which establish
  the duality between L-spaces and lattices to functors between
  $\cat{SLSpace}$ and $\cat{CLat}$.
  By abuse of notation, we denote these functors by $\Fclp$ and $\Ftop$ as well. 

  We have already seen that taking the complex algebra of a general frame
  yields a conditional lattice.
  Therefore the clopen filters of a selection L-space $\topo{X}$ form a conditional
  lattice, which we denote by $\Fclp(\topo{X})$.
  Furthermore, it follows from Lemma~\ref{lem:smor-dual} that
  for every selection morphism $f : \topo{X} \to \topo{X}'$,
  its inverse $f^{-1}$ is a conditional lattice homomorphism from $\Fclp(\topo{X}')$
  to $\Fclp(\topo{X})$.
  A routine verification shows that this gives rise to a contravariant functor
  \begin{equation*}
    \Fclp : \cat{SLSpace} \to \cat{CLat}.
  \end{equation*}
  For the converse, we extend $\Ftop$ as follows.
  
\begin{definition}\label{def:lattopo}
  Let $\amo{A} = (A, \top, \bot, \land, \lor, \cto)$ be a conditional lattice
  and let $(\Filt(\amo{A}), A, \cap, \tau_{\amo{A}})$ be the L-space dual to
  the lattice $(A, \top, \bot, \land, \lor)$.
  Then as a consequence of the duality for lattices (from Theorem~\ref{thrm:weakduality})
  we know that every clopen filter is of the form $\theta_{\amo{A}}(a)$ for some $a \in A$.
  Furthermore, define
  \begin{equation*}
    s_{\amo{A}} : \Filt(\amo{A}) \times \Fclp(\Filt(\amo{A}), A, \cap, \tau_{\amo{A}})
     \to \Fcls(\Filt(\amo{A}), A, \cap, \tau_{\amo{A}})
  \end{equation*}
  by
  \begin{equation*}
    s_{\amo{A}}(p, \theta_{\amo{A}}(a)) = \ua \{ b \in A \mid a \cto b \in p \}.
  \end{equation*}
  Then we let $\Ftop(\amo{A}) := (\Filt(\amo{A}), A, \cap, s_{\amo{A}}, \tau_{\amo{A}})$.
\end{definition}

  Note that $s_{\amo{A}}(p, \theta_{\amo{A}}(a))$ consists of all filters
  $q$ in $\Filt(\amo{A})$ such that $a \cto b \in p$ implies $b \in q$.
  Therefore we could have equivalently defined
  \begin{equation*}
    s_{\amo{A}}(p, \theta_{\amo{A}}(a))
      = \bigcap \{ \theta_{\amo{A}}(b) \mid a \cto b \in p \}.
  \end{equation*}
  We verify that $\Ftop(\amo{A})$ is a selection L-space whenever
  $\amo{A}$ is a conditional lattice.

\begin{lemma}\label{lem:Ftoplspace}
  For any conditional lattice $\amo{A} = (A, \top, \bot, \wedge, \vee, \cto)$,
  the tuple $\Ftop(\amo{A}) = (\Filt(\amo{A}), A, \cap, s_{\amo{A}}, \tau_{\amo{A}})$
  is a selection L-space.
\end{lemma}
\begin{proof}
  We know from Theorem~\ref{thrm:weakduality} that
  $(\Filt(\amo{A}), A, \cap, \tau_{\amo{A}})$ is an L-space.
  Furthermore, $s_{\amo{A}}$ is well defined because principal
  filters in an L-space are automatically closed~\cite[Lemma~2.8]{BezEA24},
  so we only need to verify~\eqref{it:S1d} to~\eqref{it:St}
  form Definition~\ref{def:condlspace}

  \medskip\noindent
  {\it Condition~\eqref{it:S1d}.}
    Recall that $A$ is the top element of $\Ftop(\amo{A})$,
    and that clopen filters are precisely the filters of the form $\theta_\amo{A}(a)$,
    for $a \in A$. Now we find
    \begin{equation*}
      s_\amo{A}(A, \theta_\amo{A}(a)) = \ua \{ b \in A \mid a \cto b \in A \} = \ua \{ A \} = \{ A \},
    \end{equation*}
    because $a \cto b \in A$ for all $a, b \in A$.
  
  \medskip\noindent
  {\it Condition~\eqref{it:S2d}.} Let $p, q \in \Filt(\amo{A})$
  be two filters of $A$ such that $p \subseteq q$.
  Then we have $\{ b \in A \mid a \cto b \in p \} \subseteq \{ b \in A \mid a \cto b \in q \}$,
  and hence
  \begin{equation*}
    s_{\amo{A}}(q, \theta_\amo{A}(a)) 
      = \ua \{ b \in A \mid a \cto b \in q \}
      \subseteq \ua\{ b \in A \mid a \cto b \in p \}
      = s_{\amo{A}}(p, \theta_\amo{A}(a)).
  \end{equation*}
  
  \medskip\noindent
  {\it Condition~\eqref{it:S3d}.}
  Let $p, q, r \in \Filt(\amo{A})$ and suppose
  $r \in s_\amo{A}(p \cap q,\theta_\amo{A}(a))$.
  Let $p' = \{ b \in A \mid a \cto b \in p \}$ and
  $q' = \{b \in A \mid a \cto b \in q \}$.
  Then $p' \in s_\amo{A}(p,\theta_\amo{A}(a))$, $q' \in
  s_\amo{A}(q,\theta_\amo{A}(a))$ and by construction
  $p' \cap q' = \{ b \in A \mid a \cto b \in p \cap q \} \subseteq r$.

  \medskip\noindent
  {\it Condition~\eqref{it:St}.}
  Let $\theta_\amo{A}(a)$ and $\theta_\amo{A}(b)$ be two clopen filters of $\Ftop(\amo{A})$.
  For any $p \in \Filt(\amo{A})$ we have
  \begin{equation*}
    s_{\amo{A}}(p, \theta_\amo{A}(a)) \subseteq \theta_\amo{A}(b)
      \iff \{ c \in A \mid a \cto c \in p \} \in \theta_\amo{A}(b)
      \iff a \cto b \in p.
  \end{equation*}
  Therefore $\theta_\amo{A}(a) \dto \theta_\amo{A}(b) = \{ p \in \Filt(\amo{A}) \mid
  s_{\amo{A}}(p, \theta_\amo{A}(a)) \subseteq \theta_\amo{A}(b) \} = \theta_\amo{A}(a \cto b)$,
  which is clopen in $\Ftop(\amo{A})$ by definition.
\end{proof}

  We now know that the $\Ftop$ defines an assignment $\cat{CLat} \to \cat{SLSpace}$ on 
  objects. We now extend this to morphisms.

\begin{lemma}
  Let $\amo{A} = (A, \top, \bot, \wedge, \vee, \cto)$ and
  $\amo{A}' = (A', \top', \bot', \wedge', \vee', \cto')$ be two conditional
  lattices, and
  $h : \amo{A} \to \amo{A}'$ a conditional lattice homomorphism.
  Then $h^{-1} : \Ftop(\amo{A}') \to \Ftop(\amo{A})$ is a selection morphism
  between selection L-spaces.
\end{lemma}
\begin{proof}
  Since $\Ftop$ extends the similarly named functor $\cat{Lat} \to \cat{LSpace}$,
  we know from Theorem~\ref{thrm:weakduality} that $h^{-1}$ is an L-morphism.
  Furthermore, it follows that $(h^{-1})^{-1}(\theta_\amo{A}(a)) = \theta_\amo{A}(h(a))$ 
  hence $(h^{-1})^{-1}$ is continuous.
  So we only have to prove that~\eqref{it:smor-1d} and~\eqref{it:smor-2d} hold.

  \medskip\noindent
  {\it Condition~\eqref{it:smor-1d}.}
  Suppose $q' \in s_{\amo{A}'}(p', (h^{-1})^{-1}(\theta_{\amo{A}}(a)))$.
  Since
  \begin{equation*}
    s_{\amo{A}'}(p', (h^{-1})^{-1}(\theta_{\amo{A}}(a)))
      = s_{\amo{A}'}(p', \theta_{\amo{A}'}(h(a)))
      = \ua \{ b' \in A' \mid h(a) \cto b' \in p' \}
  \end{equation*}
  we have $\{ b' \in A' \mid h(a) \cto b' \in p' \} \subseteq q'$.
  This implies that
  \begin{equation*}
    \{ b \in A \mid a \cto b \in h^{-1}(p') \}
      = \{ b \in A \mid h(a) \cto h(b) \in p' \}
      \subseteq h^{-1}(q'),
  \end{equation*}
  so that $h^{-1}(q') \in s_{\amo{A}}(h^{-1}(p'), \theta_{\amo{A}}(a))$
  by definition.

  \medskip\noindent
  {\it Condition~\eqref{it:smor-2d}.}
  Let $q \in s_\amo{A}(h^{-1}(p'), \theta_{\amo{A}}(a))$.
  Then $\{ b \in A \mid a \cto b \in h^{-1}(p') \} \subseteq q$.
  We have seen that
  $s_\amo{A}(h^{-1}(p'), (h^{-1})^{-1}(\theta_{\amo{A}}(a)))
    = \ua \{ b' \in A' \mid h(a) \cto b' \in p' \}$
  so to verify~\eqref{it:smor-2d} it suffices to show that
  $h^{-1}(\{ b' \in A' \mid h(a) \cto b' \in p' \}) \subseteq q$.
  To see that this is indeed the case, compute
  \begin{align*}
    h^{-1}(\{ b' \in A' \mid h(a) \cto b' \in p' \})
    &= \{ b \in A \mid h(a) \cto h(b) \in p' \} \\
    &= \{ b \in A \mid a \cto b \in h^{-1}(p') \}
    \subseteq q.
  \end{align*}
  This proves that $h^{-1}$ is a selection morphism.
\end{proof}

  We now know that setting $\Ftop(h) = h^{-1}$ yields a contravariant
  assignment from $\cat{CLat}$ to $\cat{SLSpace}$. A routine verification
  shows that this yields contravariant functor
  \begin{equation*}
    \Ftop : \cat{CLat} \to \cat{SLSpace}.
  \end{equation*}
  This brings us to the duality for conditional lattices.

\begin{theorem}  \label{thrm:duality}
  The functors $\Fclp$ and $\Ftop$ establish a dual equivalence:
  \begin{equation*}
    \cat{SLSpace} \equiv^{op} \cat{CLat}.
  \end{equation*}
\end{theorem}
\begin{proof}
  Define transformations
  $\theta : \mathcal{id}_{\cat{CLat}} \to \Fclp \circ \Ftop$
  and $\eta : \mathcal{id}_{\cat{SLSpace}} \to \Ftop \circ \Fclp$
  on components by
  \begin{equation*}
    \theta_{\amo{A}}(a) = \{ p \in \Ftop(\amo{A}) \mid a \in p \}
    \quad\text{and}\quad
    \eta_{\topo{X}}(x) = \{ a \in \Fclp(\topo{X}) \mid x \in a \}.
  \end{equation*}
  We claim that it suffices to show that
  $\theta$ and $\eta$ are isomorphisms on components.
  Since they are defined as in Theorem~\ref{thrm:weakduality},
  it follows that they are natural isomorphisms.
  Moreover, since the functors $\Fclp$ and $\Ftop$ extend the
  functors from Theorem~\ref{thrm:weakduality}, it also follows
  automatically that the triangle identities hold, that is,
  $\Fclp\eta_{\topo{X}} \circ \theta_{\Fclp\topo{X}} = \mathcal{id}(\Fclp\topo{X})$
  for all selection L-spaces $\topo{X}$,
  and $\Ftop\theta_{\amo{A}} \circ \eta_{\Ftop\amo{A}} = \mathcal{id}(\Ftop\amo{A})$
  for every conditional lattice $\amo{A}$.
  This then implies that $\Fclp$ and $\Ftop$ establish a dual equivalence.
  For the remainder of the proof, fix a conditional lattice
  $\amo{A} = (A, \top, \bot, \wedge, \vee, \cto)$
  and a selection L-space $\topo{X} = (X, \ftop, \fmeet, s, \tau)$.
  
  \medskip\noindent
  \textit{The map $\theta_{\amo{A}}$ is an isomorphism.}
    Since $\theta$ is a map between varieties of algebras, it suffices to
    show that it is a bijective homomorphism.
    We know from the duality for lattices that $\theta_{\amo{A}}$
    is an isomorphism between the underlying lattices, hence bijective.
    So we only have to show that it preserves conditional implications, meaning
    $\theta_{\amo{A}}(a \cto b) = \theta_{\amo{A}}(a) \dto \theta_{\amo{A}}(b)$, 
    where $\dto$ is defined in the dual selection L-space.
    But this was already shown in the proof of Lemma~\ref{lem:Ftoplspace}.
    So $\theta_{\amo{A}}$ is a bijective conditional lattice homomorphism,
    hence an isomorphism.
    
  \medskip\noindent
  \textit{The map $\eta_{\topo{X}}$ is an isomorphism.}
    We know from the duality for lattices that $\eta_{\topo{X}}$ is
    an isomorphism between the underlying L-spaces.
    This implies that it is a bijective L-space morphism.
    It can be shown that the isomorphisms in $\cat{SLSpace}$ are precisely
    the bijective selection morphisms, so we only have to show that $\eta_{\topo{X}}$
    satisfies~\eqref{it:smor-1d} and~\eqref{it:smor-2d}.
    
    For condition~\eqref{it:smor-1d}, let $x, y \in X$ and let $\theta_{\Fclp\topo{X}}(a)$
    be any clopen filter of $\Ftop(\Fclp(\topo{X}))$ (where $a \in \Fclp(\topo{X})$).
    Suppose $y \in s(x, \eta_{\topo{X}}^{-1}(\theta_{\Fclp\topo{X}}(a)))$.
    It follows from the triangle equalities for $\eta$ and $\theta$
    that $\eta_{\topo{X}}^{-1}(\theta_{\Fclp(\topo{X})}(a)) = a$,
    so $y \in s(x, \eta_{\topo{X}}^{-1}(\theta_{\Fclp\topo{X}}(a))) = s(x, a)$.
    We need to show that $\eta_{\topo{X}}(y) \in s_{\Fclp\topo{X}}(\eta_{\topo{X}}(x), \theta_{\Fclp\topo{X}}(a))$.
    By definition
    \begin{align*}
      s_{\Fclp\topo{X}}(\eta_{\topo{X}}(x), \theta_{\amo{A}}(a))
        &= \ua \{ b \in \Fclp\topo{X} \mid a \cto b \in \eta_{\topo{X}}(x) \} \\
        &= \ua \{ b \in \Fclp\topo{X} \mid x \in a \dto b \}
        = \ua \{ b \in \Fclp\topo{X} \mid s(x, a) \subseteq b \}.
    \end{align*}
    Since $y \in s(x, a)$, it follows that
    $\{ b \in \Fclp\topo{X} \mid s(x, a) \subseteq b \} \subseteq \eta_{\topo{X}}(y)$.
    This proves that
    $\eta_{\topo{X}}(y) \in s_{\Fclp\topo{X}}(\eta_{\topo{X}}(x), \theta_{\amo{A}}(a))$,
    as desired.
    
    For condition~\eqref{it:smor-2d},
    let $x \in X$ and let $\theta_{\Fclp\topo{X}}(a)$ be a clopen filter of
    $\Ftop(\Fclp(\topo{X}))$. Suppose $y'$ is an element of $\Ftop(\Fclp(\topo{X}))$
    such that $y' \in s_{\Fclp\topo{X}}(\eta_{\topo{X}}(x), \theta_{\Fclp\topo{X}}(a))$.
    Then
    \begin{align*}
      y'
        \supseteq \{ b \in \Fclp\topo{X} \mid a \dto b \in \eta_{\topo{X}}(x) \}
        &= \{ b \in \Fclp\topo{X} \mid x \in a \dto b \} \\
        &= \{ b \in \Fclp\topo{X} \mid s(x, a) \subseteq b \}.
    \end{align*}
    Since $\eta$ is an L-space isomorphism, it is bijective.
    Therefore $y'$ must be of the form $\eta_{\topo{X}}(y)$ for some $y \in X$.
    Now it suffices to show that $y \in s(x, \eta_{\topo{X}}^{-1}(\theta_{\Fclp\topo{X}}(a)))$.
    We have seen that
    $s(x, \eta_{\topo{X}}^{-1}(\theta_{\Fclp\topo{X}}(a))) = s(x, a)$,
    so we only need to show that $y \in s(x, a)$.
    By assumption we have $\{ b \in \Fclp\topo{X} \mid s(x, a) \subseteq b \} \subseteq \eta_{\topo{X}}(y)$,
    which implies $y \in \bigcap \{ b \in \Fclp\topo{X} \mid s(x, a) \subseteq b \} = s(x, a)$.
    Therefore $\eta$ is a selection morphism, hence an isomorphism.
    This completes the proof of the duality.
\end{proof}

\subsection{Completeness via duality}\label{sec:scf}

  We leverage the duality from Theorem~\ref{thrm:duality} to get completeness
  for $\log{CL}$ with respect to selection L-frames.
  First, we prove a completeness result for selection L-spaces.

\begin{theorem}
  Let $\Gamma$ be a set of consequence pairs.
  Then the logic $\log{CL}(\Gamma)$ is sound and complete with respect to the
  class of selection L-spaces that validate all consequence pairs in $\Gamma$.
\end{theorem}
\begin{proof}
  Soundness was proven in Theorem~\ref{thrm:framesound}. 
  For completeness, suppose that $\phi \cp \psi$
  is not derivable in $\log{CL}(\Gamma)$.
  Then by Theorem~\ref{thm:alg-sound-compl} there exists a conditional
  lattice $\amo{A}$ that validates all consequence pairs in $\Gamma$, 
  but not $\phi \cp \psi$.
  Let $\topo{X}$ be the selection L-space dual to $\amo{A}$.
  Then $\Fclp(\topo{X})$ is isomorphic to $\amo{A}$,
  so it follows from Lemma~\ref{lem:gen-cpx-summary} (and the fact that we can
  view selection L-spaces as general frames) that $\topo{X}$ validates
  all consequence pairs in $\Gamma$ but not $\phi \cp \psi$.
  This proves completeness.
\end{proof}

  In order to transfer the completeness result for selection L-spaces
  to one for selection L-frames, we need a way to go from spaces to frames.
  In normal modal logic this is easy: simply forget about the topology.
  However, when working with selection functions the situation is more
  precarious. Indeed, a topological selection function is defined only
  on clopen filters, so discarding the topology leaves us with an
  ``incomplete'' selection function. We remedy this by ``filling in''
  the gaps. This can be done in several ways, one of which is defined next.
  
\begin{definition}\label{def:top-fill-in}
  Let $\topo{X} = (X, \ftop, \fmeet, s, \tau)$ be a conditional L-space.
  Define
  \begin{equation*}
    s_1 : X \times \Filt(X, \ftop, \fmeet) \to \Filt(X, \ftop, \fmeet)
  \end{equation*}
  by
  \begin{equation*}
    s_1(x, p) = \left\{%
      \begin{array}{ll}
        s(x, p) &\text{if $p$ is a clopen filter} \\
        \{ 1 \} &\text{if $p$ is a non-clopen fitler}
      \end{array}\right.
  \end{equation*}
  We define the \emph{top fill-in} or the \emph{$\kappa_1$-fill-in}
  of $\topo{X}$ as $\kappa_1\topo{X} := (X, \ftop, \fmeet, s_1)$.
\end{definition}

  Evidently, every clopen valuation for $\topo{X}$ is also a valuation
  for $\kappa_1\topo{X}$. Moreover, if $V$ is a clopen valuation then
  formulas in $(\topo{X}, V)$ are
  interpreted in the same way as in $(\kappa_1\topo{X}, V)$, because
  only the action of $s_1$ on clopen filters is used the determine
  the interpretation of formulas. Therefore:

\begin{lemma}\label{lem:kappa1}
  Let $\topo{X}$ be a selection L-space such that
  $\topo{X} \not\Vdash \phi \cp \psi$.
  Then $\kappa_1\topo{X} \not\Vdash \phi \cp \psi$.
\end{lemma}

  Using this, we obtain completeness of $\log{CL}$ with respect selection L-frames.

\begin{theorem}\label{thm:basic-completeness}
  The logic $\log{CL}$ is sound and complete with respect to the class
  of all selection L-frames.
\end{theorem}
\begin{proof}
  Soundness was proven in Theorem~\ref{thrm:framesound}.
  For completeness, suppose $\phi \cp \psi$
  is not derivable in $\log{CL}$.
  Then there exists a selection L-space $\topo{X}$ that does not
  validate $\phi \cp \psi$, so that by Lemma~\ref{lem:kappa1}
  $\kappa_1\topo{X}$ does not validate $\phi \cp \psi$.
\end{proof}

  While the top fill-in allows us to prove completeness for $\log{CL}$,
  we do not automatically get completeness for logics of the form $\log{CL}(\Gamma)$.
  Indeed, while $\topo{X}$ may validate all consequence pairs in $\Gamma$,
  this need not be the case for $\kappa_1\topo{X}$.
  In the next section we explore how we can modify the fill-in
  to obtain completeness for various extensions of $\log{CL}$.

\section{Extensions}\label{sec:extensions}

  To showcase the versatility of the semantic framework of sub-sub-intuitionistic,
  we consider various extension of the basic logic
  $\log{CL}$ with axioms (that is, with consequence pairs).
  We provide completeness results for each of these, as well as certain
  combinations of axioms.
  We use a strategy similar to that in Section~\ref{sec:scf} with two
  additional steps:
  \begin{enumerate}
    \item For each axiom, we find a frame condition corresponding to validity
          of the axiom. We prove these correspondence results with respect to
          general frames and them to selection L-frames and -spaces.
    \item We show that validity of the axiom under consideration
          (or the corresponding frame condition) is preserved when moving
          from a selection L-space to a selection L-frame.
  \end{enumerate}
  
  This reveals an interesting phenomenon: there may be several ways to extend
  a topological selection function to a selection function that is defined on
  all filters, and some of them may not preserve validity of the axiom
  under consideration.
  In Section~\ref{subsec:non-top} we shall see an example
  of an axiom that is not preserved under the top fill-in
  (from Definition~\ref{def:top-fill-in}. That is, there exists a selection
  L-space $\topo{X}$ validating the axiom while $\kappa_1\topo{X}$ does not.
  We can still obtain completeness, but we have to use a different type
  of fill-in.
  
  We use the completeness results for extensions of $\log{CK}$ to elucidate
  the connection between sub-sub-intuitionistic logic and
  intuitionistic conditional logic~\cite{Wei19a} in
  Section~\ref{subsec:ssi-vs-int-cond}, and to derive an alternative
  (semilattice-based) semantics for intuitionistic logic
  in Section~\ref{subsec:int-sl}.
 
  Where possible, the naming of the axioms we consider follows that of
  classical counterparts from~\cite{UntSch14}.
  We refer to \cite{UntSch14} for a discussion of the axioms.
  All correspondence results are given in the appendix.

\subsection{Some simple completeness results}\label{subsec:top-fill-in}

  We begin our investigation of extensions of $\log{CL}$ with four axioms
  whose completeness we can prove in a way similar to
  Theorem~\ref{thm:basic-completeness}, using the top fill-in.
  As announced, to achieve this we give correspondence conditions for each
  of the axioms, and we show that these are preserved by taking the
  top fill-in of a selection L-space, a property
  we call \emph{$\kappa_1$-persistence}.

  Consider the following axioms:
  \begin{multicols}{2}
  \begin{enumerate}\itemsep=0pt
    \renewcommand{\labelenumi}{(\theenumi) }
    \renewcommand{\theenumi}{$\mathsf{refl}$}
    \item \label{ax:refl}
          $\top \cp p \cto p$
    \renewcommand{\theenumi}{$\mathsf{cond}$}
    \item \label{ax:cond}
          $p \cp (\top \cto p)$
    \renewcommand{\theenumi}{$\mathsf{veq}$}
    \item \label{ax:veq}
          $p \cp q \cto p$
    \renewcommand{\theenumi}{$\mathsf{cs}$}
    \item \label{ax:cs}
          $p \wedge q \cp p \cto q$
  \end{enumerate}
  \end{multicols}

  We find the following frame correspondence conditions.

\begin{proposition}\label{prop:triv-fill-corr}
  A general frame $\mo{G} = (X, \ftop, \fmeet, s, A)$ validates
  $\log{ax} \in \{ \axref{ax:refl}, \axref{ax:cond}, \axref{ax:veq}, \axref{ax:cs} \}$
  if and only if it satisfies
  the frame condition \textup{($\log{ax}$-corr)} given by:
  \begin{enumerate}\itemsep=0pt
    \renewcommand{\labelenumi}{\textup{(\theenumi)} }
          \renewcommand{\theenumi}{\textup{\ref{ax:refl}-corr}}
    \item \label{eq:refl-corr}
          \makebox[12em][l]{$s(x, a) \subseteq a$} for all $x \in X$ and $a \in A$;
          \renewcommand{\theenumi}{\textup{\ref{ax:cond}-corr}}
    \item \label{eq:cond-corr}
          \makebox[12em][l]{$x \in a$ implies $s(x, X) \subseteq a$} for all $x \in X$ and $a \in A$;
          \renewcommand{\theenumi}{\textup{\ref{ax:veq}-corr}}
    \item \label{eq:veq-corr}
          \makebox[12em][l]{$x \in a$ implies $s(x, b) \subseteq a$} for all $x \in X$ and $a, b \in A$;
          \renewcommand{\theenumi}{\textup{\ref{ax:cs}-corr}}
    \item \label{eq:cs-corr}
          \makebox[12em][l]{$x \in a \cap b$ implies $s(x, a) \subseteq b$} for all $x \in X$ and $a, b \in A$.
  \end{enumerate}
  Furthermore, if $\mo{G}$ is full or descriptive then
  \eqref{eq:veq-corr} and~\eqref{eq:cs-corr} are equivalent to, respectively,
  \begin{enumerate}\itemsep=0pt
    \renewcommand{\labelenumi}{\textup{(\theenumi)} }
          \renewcommand{\theenumi}{\textup{\ref{ax:veq}-corr$'$}}
    \item \label{eq:veq-corr-fd}
          \makebox[12em][l]{$s(x, a) \subseteq \ua x$} for all $x \in X$ and $a \in A$;
          \renewcommand{\theenumi}{\textup{\ref{ax:cs}-corr$'$}}
    \item \label{eq:cs-corr-fd}
          \makebox[12em][l]{$x \in a$ implies $s(x, a) \subseteq \ua x$} for all $x \in X$ and $a \in A$;
  \end{enumerate}
\end{proposition}
  
  Next, we show that satisfaction of the correspondence conditions
  is preserved when moving from selection L-space $\topo{X}$
  to the selection L-frame $\kappa_1\topo{X}$.
  In the context of normal modal logic, this property is sometimes called 
  \emph{d-persistence}~\cite[Definition~5.84]{BlaRijVen01}.
  Correspondingly, we call the move from spaces to frames \emph{$\kappa_1$-persistence}.

\begin{lemma}\label{lem:triv-fill-canonicity}
  The axioms $\axref{ax:refl}$, $\axref{ax:cond}$, $\axref{ax:veq}$
  and $\axref{ax:cs}$ are all $\kappa_1$-persistent.
\end{lemma}
\begin{proof}
  Let $\topo{X} = (X, \ftop, \fmeet, s, \tau)$ be a selection L-space.
  Suppose $\topo{X}$ satisfies~\eqref{eq:refl-corr}.
  Then we have $s(x, a) \subseteq a$ for all $x \in X$ and all
  clopen filters $a$ of $\topo{X}$. We need to prove that
  $s_1(x, a) \subseteq a$ for all filters $a$ of $(X, \ftop, \fmeet)$.
  If $a$ is a clopen filter, then we have $s_1(x, a) = s(x, a) \subseteq a$.
  If not, then $s_1(x, a) = \{ \ftop \} \subseteq a$ because all filters
  contain $\ftop$. So $\kappa_1\topo{X}$ satisfies~\eqref{eq:refl-corr}.
  This proves that $\axref{ax:refl}$ is $\kappa_1$-persistent.
  The cases for~$\axref{ax:cond}$ and $\axref{ax:veq}$ are similar,
  using \eqref{eq:veq-corr-fd} for the latter.
  
  To prove $\kappa_1$-persistence of $\axref{ax:cs}$ we use the
  correspondence condition~\eqref{eq:cs-corr-fd}.
  Let $\topo{X} = (X, \ftop, \fmeet, s, \tau)$ be a conditional L-space
  that satisfies~\eqref{eq:cs-corr-fd}.
  We show that $x \in a$ implies $s_1(x, a) \subseteq \ua x$ for every
  filter $a$ of $(X, \ftop, \fmeet)$.
  Let $a$ be such a filter and $x \in a$.
  If $a$ is a clopen filter of $\topo{X}$, then we have
  $s_1(x, a) = s(x, a) \subseteq \ua x$.
  If not, then $s_1(x, a) = \{ \ftop \} \subseteq \ua x$.
\end{proof}

\begin{theorem}\label{thm:fill-1-compl}
  Let $\Ax \subseteq \{ \axref{ax:refl}, \axref{ax:cond}, \axref{ax:veq}, \axref{ax:cs} \}$. Then the logic
  $\log{CL}(\{ \Ax \})$ is sound and complete with respect to the
  class of selection L-frames that satisfy \textup{($\log{ax}$-corr)}
  for each $\log{ax} \in \Ax$.
\end{theorem}
\begin{proof}
  Soundness follows from Proposition~\ref{prop:triv-fill-corr}.
  Completeness can be shown as in Theorem~\ref{thm:basic-completeness},
  using the fact that Lemma~\ref{lem:triv-fill-canonicity} implies that
  a selection L-space validating $\Ax$ gives rise to a selection
  L-frame validating $\Ax$.
\end{proof}

  We have seen four axioms that are all $\kappa_1$-persistent.
  But there might be other ways to extend a topological selection function to
  a selection function that preserves the axioms. For example, if
  $\topo{X} = (X, \ftop, \fmeet, s, \tau)$ is a selection L-space we could
  define
  \begin{equation*}
    s_r(x, a) = \left\{
      \begin{array}{ll}
        s(x, a) &\text{if $a \in \Fclp\topo{X}$} \\
        a &\text{otherwise}
      \end{array}\right.
  \end{equation*}
  and define $\kappa_r\topo{X} = (X, \ftop, \fmeet, s_r)$.
  Then $\axref{ax:refl}$, $\axref{ax:cond}$ and $\axref{ax:veq}$ are
  $\kappa_r$-persistent as well.
  
  In the next subsection we will see some axioms that are persistent
  for any transformation from selection L-spaces to selection L-frames
  obtained by extending the selection function and forgetting the topology.

\subsection{Fill-in agnostic axioms}\label{subsec:fill-in-agn}

  We say that an axiom is \emph{all-persistent} if its validity is preserved by
  any extension of a topological selection function of a selection L-space to a
  selection function for the underlying L-frame. Such axioms are nice because
  they can be combined with other axioms that require more specialised fill-ins.
  Moreover, it turns out that every $\cto$-free axiom is all-persistent.
  This affords us the possibility to add, for example, distributivity or
  modularity to a large range of logics.
  
  We consider the following axioms and class of axioms:
  \begin{multicols}{2}
  \begin{enumerate}\itemsep=0pt
    \renewcommand{\labelenumi}{(\theenumi) }
    \renewcommand{\theenumi}{$\mathsf{det}$}
    \item \label{ax:det}
          $(\top \cto p) \cp p$
    \renewcommand{\theenumi}{$\mathsf{expl}$}
    \item \label{ax:expl}
          $\top \cp (\bot \cto p)$
    \renewcommand{\theenumi}{$\mathsf{prop}$}
    \item \label{ax:prop}
          any $\cto$-free consequence pair
  \end{enumerate}
  \end{multicols}

\begin{proposition}\label{prop:agnostic-fill-corr}
  Let $\mo{G} = (X, \ftop, \fmeet, s, A)$ be a general conditional L-frame.
  Then $\mo{G}$ validates $\axref{ax:det}$ if and only if it satisfies
  \begin{enumerate}\itemsep=0pt
    \renewcommand{\labelenumi}{\textup{(\theenumi)} }
          \renewcommand{\theenumi}{\textup{\ref{ax:det}-corr}}
    \item \label{eq:det-corr}
          \makebox[12em][l]{$s(x, X) \subseteq a$ implies $x \in a$}
          for all $x \in X$ and $a \in A$;
  \end{enumerate}
  and $\mo{G}$ validates $\axref{ax:expl}$ if and only if it satisfies
  \begin{enumerate}\itemsep=0pt
    \renewcommand{\labelenumi}{\textup{(\theenumi)} }
          \renewcommand{\theenumi}{\textup{\ref{ax:expl}-corr}}
    \item \label{eq:expl-corr}
          \makebox[12em][l]{$s(x, \{ \ftop \}) = \{ \ftop \}$} for all $x \in X$.
  \end{enumerate}
  Moreover, if $\mo{G}$ is full or descriptive then it validates
  $\axref{ax:det}$ if and only if it satisfies
  \begin{enumerate}\itemsep=0pt
    \renewcommand{\labelenumi}{\textup{(\theenumi)} }
          \renewcommand{\theenumi}{\textup{\ref{ax:det}-corr$'$}}
    \item \label{eq:det-corr-fd}
          \makebox[12em][l]{$x \in s(x, X)$} for all $x \in X$.
  \end{enumerate}
\end{proposition}

\begin{lemma}\label{lem:fill-in-agn}
  The axioms $\axref{ax:det}$, $\axref{ax:expl}$ and $\axref{ax:prop}$
  are all-consistent.
\end{lemma}
\begin{proof}
  Let $\topo{X} = (X, \ftop, \fmeet, s, \tau)$ be a selection L-space
  and let $\hat{s}$ be a selection function for $(X, \ftop, \fmeet)$
  such that $\hat{s}(x, a) = s(x, a)$ for all $x \in X$ and
  $a \in \Fclp(\topo{X})$.
  Write $\hat{\kappa}\topo{X} = (X, \ftop, \fmeet, \hat{s})$.
  If $\topo{X}$ validates $\axref{ax:det}$, then it satisfies
  $x \in s(x, X)$ for all $x \in X$.
  Since $X$ is a clopen filter of $\topo{X}$,
  we have $\hat{s}(x, X) = s(x, X) \ni x$ for all $x \in X$,
  so $\hat{\kappa}\topo{X}$ satisfies \eqref{eq:det-corr-fd}
  hence validates $\axref{ax:det}$.
  The case for $\axref{ax:expl}$ is similar, making use of the fact that
  $\{ 1 \}$ is a clopen filter by definition.
  
  The case for $\axref{ax:prop}$ follows from~\cite[Lemma~3.21]{BezEA24}
  and the fact that the interpretation of $\cto$-free formulas does not
  rely on the (topological) selection function.
\end{proof}

\begin{theorem}\label{thm:fill-in-agn}
  Let $\Ax \subseteq \{ \axref{ax:refl}, \axref{ax:cond}, \axref{ax:veq},
  \axref{ax:cs}, \axref{ax:det}, \axref{ax:expl}, \axref{ax:prop} \}$.
  Then $\log{CL}(\Ax)$ is sound and complete with respect to the class
  of conditional L-frames that validate $\log{ax}$
  for each $\log{ax} \in \Ax$.
\end{theorem}
\begin{proof}
  This is similar to Theorem~\ref{thm:fill-1-compl},
  using Lemma~\ref{lem:fill-in-agn} for the new axioms.
\end{proof}

  Like Theorem~\ref{thm:fill-1-compl}, Theorem~\ref{thm:fill-in-agn} can also
  be formulated in terms of the a class of frames satisfying
  ($\log{ax}$-corr) for each $\log{ax} \in \Ax$,
  where for each $\cto$-free axiom has a frame
  correspondent~\cite[Theorem~3.32]{BezEA24}.

\subsection{Sub-sub-intuitionistic logic versus intuitionistic conditional logic}
\label{subsec:ssi-vs-int-cond}

  We use Theorem~\ref{thm:fill-in-agn} to investigate the extension of
  $\log{CL}$ with a specific instance of $\axref{ax:prop}$,
  namely the distributivity axiom
  $p \wedge (q \vee r) \cp (p \wedge q) \vee (p \wedge r)$.
  This gives rise to a minimal system of conditional logic over a
  \emph{distributive} positive base, so we will call the resulting
  logic $\log{PCL}$ for positive conditional logic.
  This logic is of particular interest because of the way it relates
  to the intuitionistic conditional logic $\log{ICK}$,
  introduced by Weiss~\cite{Wei19a}:
  intuitionistic conditional logic is a conservative extension of $\log{PCL}$.
  Before proceeding, we recall the definition of
  intuitionistic conditional logic.

\begin{definition}
  Let $\mathbf{Int}_{\cto}(\Prop)$ be the language given by the grammar
  \begin{equation*}
    \phi ::= p \in \Prop \mid \top \mid \bot \mid \phi \wedge \phi \mid \phi \vee \phi \mid \phi \to \phi \mid \phi \cto \phi,
  \end{equation*}
  where $\Prop$ is some arbitrary but fixed set of proposition letters.
  Intuitionistic conditional logic~\cite{Wei19a} is the logic
  $\log{ICK}$ obtained by extending an axiomatisation of intuitionistic
  logic with the axioms
  \begin{align*}
    p \cto \top, && (p \cto (q \wedge r)) \to\ ((p \cto q) \wedge (p \cto r), \\
    && ((p \cto q) \wedge (p \cto r) \to\ (p \cto (q \wedge r)),
  \end{align*}
  the congruence rules, and closing it under modus ponens and 
  uniform substitution.
  The algebraic semantics of $\log{ICK}$ is given by Heyting algebras
  with a binary operator $\cto$ satisfying the condition above.
  These are called \emph{conditional Heyting algebras}.
\end{definition}

  We assume that $\log{PCL}$ and $\log{ICK}$ are defined over the same set
  of proposition letters.
  In order to describe the frames for $\log{PCL}$, we define what it means
  for a semilattice to be distributive.

\begin{definition}\label{def:Lfrm-distr}
  A semilattice $(X, \ftop, \fmeet)$ is called \emph{distributive}
  if for all $x, y, z \in X$ such that $x \fmeet y \fleq z$ we can find
  $u, v \in X$ such that $x \fleq u$ and $y \fleq v$ and $z = u \fmeet v$.
\end{definition}

  Now as a consequence of~\cite[Example~3.33]{BezEA24}
  and Theorem~\ref{thm:fill-in-agn} we have:

\begin{theorem}
  The logic $\log{PCL}$ is sound and complete with respect to the
  class of distributive selection L-frames.
\end{theorem}

  We aim to prove that $\log{ICK}$ is a conservative extension
  of $\log{PCL}$, in the sense that a consequence pair $\phi \cp \psi$ is
  derivable in $\log{PCL}$ if and only if the formula $\phi \to \psi$ is
  derivable in $\log{ICK}$.
  We prove this algebraically by making use of the fact that the complex algebra of any
  distributive selection L-frame can be equipped with a Heyting implication.
  This yields a conditional Heyting algebra~\cite[Definition~3]{Wei19a}.

\begin{lemma}\label{lem:dist-cpx-frm}
  Let $\mo{X} = (X, \ftop, \fmeet, s)$ be a distributive selection L-frame.
  Then the complex algebra $\mo{X}^+$ satisfies the
  frame distributivity law, that is, we have
  \begin{equation*}
    p \cap \bigGenFil \{ q_i \mid i \in I \}
      = \bigGenFil \{ p \cap q_i \mid i \in I \}
  \end{equation*}
  for every filter $p \in \mo{X}^+$ and
  family of filters $\{ q_i \mid i \in I \} \subseteq \mo{X}^+$.
\end{lemma}
\begin{proof}
  We have $q_i \subseteq \bigGenFil \{ q_i \mid i \in I \}$ for all $i \in I$,
  so that each $p \cap q_i$ is contained in
  $p \cap \bigGenFil \{ q_i \mid i \in I \}$.
  Therefore $p \cap \bigGenFil \{ q_i \mid i \in I \}$ is an upper bound for
  the disjunction on the right hand side, which proves the right-to-left inclusion.
  
  For the converse, suppose $x \in p \cap \bigGenFil \{ q_i \mid i \in I \}$.
  Then $x \in p$ and $x \in \bigGenFil \{ q_i \mid i \in I \}$, and by definition
  of $\bigGenFil$ we can find a finite $I' \subseteq I$ such that
  $x \in \bigGenFil \{ q_i \mid i \in I' \}$ (see e.g.~\cite[Section~2.1]{BezEA24}).
  Distributivity of $\mo{X}$ implies distributivity of $\mo{X}^+$, so that we find
  \begin{equation*}
    x \in p \cap \bigGenFil \{ q_i \mid i \in I' \}
      = \bigGenFil \{ p \cap q_i \mid i \in I' \}
      \subseteq \bigGenFil \{ p \cap q_i \mid i \in I \}.
  \end{equation*}
  This proves the inclusion from left to right.
\end{proof}

\begin{theorem}
  Let $\phi, \psi \in \mathbf{CL}$ be two formulas and $\phi \cp \psi$ a
  consequence pair.
  Then
  \begin{equation*}
    \log{PCL} \vdash \phi \cp \psi
      \iff \log{ICK} \vdash \phi \to \psi.
  \end{equation*}
\end{theorem}
\begin{proof}
  If $\log{PCL} \vdash \phi \cp \psi$, then $\phi \cp \psi$ is valid on
  all distributive conditional lattices. In particular, this means that
  for every distributive conditional lattice $\amo{A}$ and every assignment
  $\sigma$ of the proposition letters, we have
  $\llp \phi \rrp_{(\amo{A}, \sigma)} \leq \llp \psi \rrp_{(\amo{A}, \sigma)}$.
  In particular, this holds for all conditional Heyting algebras,
  which implies that $\phi \to \psi$ is valid on all conditional Heyting algebras.
  Therefore we have $\log{ICK} \vdash \phi \to \psi$~\cite[Theorem~5]{Wei19a}.
  
  For the converse direction, we argue by contrapositive.
  Suppose that $\phi \cp \psi$ is not derivable in $\log{PCL}$.
  Then there exists a distributive selection L-frame $\mo{X}$ that does not
  validate it. Therefore its complex algebra $\mo{X}^+$ does not validate
  $\phi \cp \psi$, so there exists an assignment $\sigma$ of the proposition
  letters to the complex algebra such that
  $\llp \phi \rrp_{(\mo{X}^+, \sigma)} \not\leq \llp \psi \rrp_{(\mo{X}^+, \sigma)}$.
  As a consequence of Lemma~\ref{lem:dist-cpx-frm} the distributive lattice
  underlying $\mo{X}^+$ is a frame, hence carries a Heyting algebra
  structure~\cite[Proposition 3.10.2]{Vic89}.
  Therefore $\mo{X}^+$ is a conditional Heyting algebra.
  The residuation property then implies that
  $\top_{\mo{X}^+} \not\leq \llp \phi \rrp \to \llp \psi \rrp = \llp \phi \to \psi \rrp$,
  so that $\mo{X}^+$ is a conditional Heyting algebra invalidating $\phi \to \psi$.
  This implies that $\phi \to \psi$ is not derivable in $\log{ICK}$.
\end{proof}

\subsection{Two more extensions and the principal fill-in}\label{subsec:non-top}

  We now consider the extension of $\log{CL}$ with any of the following
  axioms:
  \begin{multicols}{2}
  \begin{enumerate}\itemsep=0pt
    \renewcommand{\labelenumi}{(\theenumi) }
    \renewcommand{\theenumi}{$\mathsf{pnp}$}
    \item \label{ax:pnp}
          $p \wedge (p \cto \bot) \cp \bot$
    \renewcommand{\theenumi}{$\mathsf{mp}$}
    \item \label{ax:mp}
          $p \wedge (p \cto q) \cp q$
  \end{enumerate}
  \end{multicols}
  \noindent
  The axiom $\axref{ax:pnp}$ reflects the idea that whenever $p$ is true at some world,
  then it can not at the same time imply falsum. Intuitively this means
  that $p \cto \bot$ behaves somewhat like a negation of $p$, albeit a very
  weak one.
  The axiom $\axref{ax:mp}$ can be viewed as a local version of modus
  ponens.
  
  From a technical point of view these axioms are interesting because they
  are not $\kappa_1$-persistent. This is illustrated by the following example.
  
\begin{example}\label{exm:kup-no-k1}
  Let $\mb{N}_{\infty} = \mb{N} \cup \{ \infty \}$ be the set of natural
  numbers with an additional point $\infty$, and let $\tau$ be the
  collection of subsets of $\mb{N}$ together with the cofinite
  subsets of $\mb{N}_{\infty}$ that contain $\infty$.
  Then $(\mb{N}_{\infty}, \tau)$ is compact (it is the one-point compactification
  of $\mb{N}$ equipped with the discrete topology.)
  
  Order $\mb{N}_{\infty}$ by the reverse natural ordering, that is,
  we let $n \fleq m$ if $n = \infty$ or $n, m \in \mb{N}$ and $m \leq n$.
  Then $\mb{N}_{\infty}$ forms a semilattice where
  $x \fmeet y$ is simply the maximum of the two.
  It is easy to see that $(\mb{N}_{\infty}, 0, \fmeet)$ is a distributive
  semilattice.
  The clopen filters of $(\mb{N}_{\infty}, 0, \fmeet, \tau)$ are
  precisely the subsets of the form
  ${\uparrow}_{\fleq}n := \{ x \in \mb{N}_{\infty} \mid n \fleq x \} = \{ 0, \ldots, n \}$.
  As a consequence, the collection of clopen filters is closed
  under $\genFil$, because $({\uparrow}_{\fleq}n) \genFil ({\uparrow}_{\fleq}m) = {\uparrow}_{\fleq}(n \fmeet m)$.
  Finally, $(\mb{N}_{\infty}, 0, \fmeet, \tau)$ satisfies the
  HMS separation axiom because if $n \not\fleq m$ then
  ${\uparrow}_{\fleq}n$ is a clopen filter containing $n$ but not $m$.
  So $(\mb{N}_{\infty}, 0, \fmeet, \tau)$ is an L-space.
  
  Define a selection function $s$ by
  $s(n, a) = {\uparrow}n$ for all $n \in \mb{N}_{\infty}$ and clopen filters $a$.
  Then the conditional L-space validates \axref{ax:pnp} and \axref{ax:mp}
  because it satisfies the correspondence conditions from
  Proposition~\ref{prop:refl-mp-veq-corr}.
  However, if we transform this to a conditional L-frame
  $(\mb{N}_{\infty}, 0, \fmeet, s_1)$, using the selection function
  $s_1$ as defined in Lemma~\ref{lem:triv-fill-canonicity} then this is no
  longer the case.
  For example, take the number $2 \in \mb{N}_{\infty}$ and
  the (non-clopen) filter $\mb{N} \subseteq \mb{N}_{\infty}$.
  Then we have $2 \in \mb{N}$ while $2 \notin \{ 0 \} = s_1(2, \mb{N})$.
\end{example}
  
  Another reason for extending $\log{CL}$ with the axioms above
  is that they give rise to logics that
  lie between $\log{CL}$ and intuitionistic logic.
  In fact, they are closely related to the equational definition of
  Heyting algebras~\cite[Section~II.1, Example 11]{BurSan81},
  and we shall see in the next section
  that extending $\log{CL}$ with these two axioms and $\axref{ax:refl}$
  gives rise to intuitionistic logic.
  We move on to the correspondence results.

\begin{proposition}\label{prop:refl-mp-veq-corr}
  Let $\mo{X} = (X, \ftop, \fmeet, s, A)$ be a general conditional L-frame.
  Then $\mo{X}$ validates $\log{ax} \in \{ \axref{ax:mp}, \axref{ax:veq} \}$
  if and only if it satisfies \textup{($\log{ax}$-corr)}, given by:
  \begin{enumerate}\itemsep=0pt
    \renewcommand{\labelenumi}{\textup{(\theenumi)} }
          \renewcommand{\theenumi}{\textup{\ref{ax:pnp}-corr}}
    \item \label{eq:pnp-corr}
          \makebox[17em][l]{if $x \in a$ and $s(x, a) = \{ \ftop \}$ then $x = \ftop$}
          for all $x \in X$ and $a \in A$;
          \renewcommand{\theenumi}{\textup{\ref{ax:mp}-corr}}
    \item \label{eq:mp-corr}
          \makebox[17em][l]{if $x \in a$ and $s(x,a) \subseteq b$ then $x \in b$}
          for all $x \in X$ and $a,b \in A$.
  \end{enumerate}
  Moreover, if $\mo{G}$ is full or descriptive then it validates
  $\axref{ax:mp}$ if and only if it satisfies
  \begin{enumerate}\itemsep=0pt
    \renewcommand{\labelenumi}{\textup{(\theenumi)} }
          \renewcommand{\theenumi}{\textup{\ref{ax:mp}-corr$'$}}
    \item \label{eq:mp-corr-fd}
          \makebox[17em][l]{$x \in a$ implies $x \in s(x, a)$}
          for all $x \in X$ and $a \in A$.
  \end{enumerate}
\end{proposition}

  As witnessed by Example~\ref{exm:kup-no-k1}, we cannot use
  $\kappa_1$-persistence to derive completeness results for the extension
  of $\log{CL}$ with~$\axref{ax:pnp}$, $\axref{ax:mp}$.
  Instead, we use the \emph{principal fill-in}, or \emph{$\kappa_{\uparrow}$-fill-in},
  and $\kappa_{\uparrow}$-persistence, defined next.

\begin{definition}
  Let $\topo{X} = (X, \ftop, \fmeet, s, \tau)$ be a conditional L-space.
  Define 
  \begin{equation*}
    s_{\ua}(x, a) = \left\{
      \begin{array}{ll}
        s(x, a) &\text{if $a$ is a filter that is clopen in $\topo{X}$} \\
        \ua x   &\text{if $a$ is a non-clopen filter}
      \end{array}
    \right.
  \end{equation*}
  and let $\kappa_{\ua}\topo{X} := (X, \ftop, \fmeet, s_{\ua})$.
  A consequence pair $\phi \cp \psi$ is called \emph{$\kappa_{\ua}$-persistent}
  if $\topo{X} \Vdash \phi \cp \psi$ implies $\kappa_{\ua}\topo{X} \Vdash \phi \cp \psi$,
  for every selection L-space $\topo{X}$.
\end{definition}

  To see that $s_{\ua}$ does indeed define a selection function on $(X, \ftop, \fmeet)$
  we need to
  verify (\ref{it:sf-top}, \ref{it:sf-up}, \ref{it:sf-filt}) for non-clopen $a$
  (the case for clopen $a$ it taken care of by the fact that $s$ is a topological
  selection function). The first two are immediate,
  and the last follows from the fact that $\ua x \genFil \ua y = \ua(x \fmeet y)$.

\begin{lemma}
  The axioms $\axref{ax:pnp}$, $\axref{ax:mp}$ and $\axref{ax:veq}$
  are $\kappa_{\uparrow}$-persistent.
\end{lemma}
\begin{proof}
  Let $\topo{X} = (X, \ftop, \fmeet, s, \tau)$ be a selection L-space
  and suppose that it validates $\axref{ax:pnp}$.
  We show that $\kappa_{\ua}\topo{X}$ satisfies \eqref{eq:pnp-corr}.
  Let $a$ be any filter of $\kappa_{\ua}\topo{X}$ and $x \in a$
  such that $s_{\ua}(x, a) = \{ 1 \}$.
  If $a$ is clopen then we have $x = \ftop$ because $\topo{X}$ satisfies
  \eqref{eq:pnp-corr}. If not, then we have
  $s_{\ua}(x, a) = \ua x \subseteq \{ \ftop \}$, which also forces
  $x = \ftop$. So $\kappa_{\ua}\topo{X}$ satisfies \eqref{eq:pnp-corr},
  hence validates $\axref{ax:pnp}$.
  
  The cases for $\axref{ax:mp}$ and $\axref{ax:veq}$ can be handled similarly.
\end{proof}

  We arrive at a completeness theorem akin to Theorem~\ref{thm:fill-1-compl}.

\begin{theorem}\label{thm:mp-veq-compl}
  Let $\Ax \subseteq \{ \axref{ax:pnp}, \axref{ax:mp}, \axref{ax:veq}, \axref{ax:det}, \axref{ax:expl}, \axref{ax:prop} \}$.
  Then the logic $\log{CL}(\Ax)$ is sound and complete with respect to
  the class of conditional L-frames that validate $\log{ax}$
  for each $\log{ax} \in \Ax$.
\end{theorem}

\subsection{Semilattice semantics for intuitionistic logic}\label{subsec:int-sl}

  Lastly, we consider the extension of $\log{CL}$ with the axioms
  $\axref{ax:refl}$, $\axref{ax:mp}$ and $\axref{ax:veq}$.
  This is particularly interesting, because the algebraic semantics of
  $\log{CL}(\{ \axref{ax:refl}, \axref{ax:mp}, \axref{ax:veq} \})$
  is given by the class of Heyting algebras.
  Therefore we obtain an alternative semantics for intuitionistic logic
  which is based on semilattices.
  
  In fact, it will turn out that we can forget about the selection function,
  and use distributive semilattices as a sound and complete semantics for
  intuitionistic logic!
  This simplifies existing semantics for intuitionistic logic from the literature
  that use (semi)lattices, such as those in \cite{Kav24,Kom86}.

\begin{proposition}\label{prop:CL-refl-mp-veq-HA}
  The algebraic semantics of
  $\log{CK}(\{ \axref{ax:refl}, \axref{ax:mp}, \axref{ax:veq} \})$
  is given by the variety of Heyting algebras.
\end{proposition}
\begin{proof}
  The inequalities induced by the additional axioms are valid in very Heyting algebra
  and imply the those
  equationally defining Heyting algebras, such as those in
  \cite[Section~II.1, Example 11]{BurSan81}.
\end{proof}

\begin{lemma}\label{lem:sf-int-corr}
  Let $\mo{X} = (X, \ftop, \fmeet, s, A)$ be a full or descriptive
  general conditional L-frame.
  Then $\mo{X}$ validates $\axref{ax:refl}, \axref{ax:mp}$ and $\axref{ax:veq}$
  if and only if 
  \begin{equation}\label{eq:int-corr}\tag{\ref{ax:refl}-\ref{ax:mp}-\ref{ax:veq}-corr}
    s(x, a) = a \cap \ua x \quad\text{for all $x \in X$ and $a \in A$}.
  \end{equation}
\end{lemma}
\begin{proof}
  It suffices to show that $s(x, a) = a \cap \ua x$ if and only if
  \eqref{eq:refl-corr}, \eqref{eq:veq-corr-fd} and \eqref{eq:mp-corr-fd}
  are satisfied for $x$ and $a$.
  If $s(x, a) = a \cap \ua x$, then it is easy to see that
  each of these holds.
  Conversely, the first and last conditions imply
  $s(x, a) \subseteq a \cap \ua x$.
  For the reverse inclusion, if $y \in a \cap \ua x$ then $y \in a$ and $x \fleq y$
  so that $y \in s(y, a) \subseteq s(x, a)$
  by~\eqref{eq:mp-corr-fd} and~\eqref{it:sf-up} respectively.
\end{proof}

  We characterise exactly which L-frames we can equip with such
  a selection function.

\begin{lemma}\label{lem:sf-distr}
  Let $(X, \ftop, \fmeet)$ be an L-frame and define a selection function $s$ by
  $s(x, p) = p \cap \ua x$.
  Then $s$ is a selection function if and only if $(X, \ftop, \fmeet)$ is
  distributive.
\end{lemma}
\begin{proof}
  Suppose $s$ is a selection function and $x, y, z \in X$ are such that
  $x \fmeet y \fleq z$. Then $z \in s(x \fmeet y, \ua z)$ so by Lemma~\ref{lem:sf-genfil}
  $z \in s(x, \ua z) \genFil s(y, \ua z)$.
  Therefore we can find $u \in s(x, \ua z)$ and $v \in s(y, \ua z)$ such that
  $u \fmeet v \fleq z$. Since $s(x, \ua z) = \ua z \cap \ua x$
  we find $x \fleq u$ and $z \fleq u$, and similarly $y \fleq v$ and $z \fleq v$.
  This implies $u \fmeet v \fleq z \fleq u \fmeet v$, so that $u \fmeet v = z$.
  This proves distributivity.
  
  For the converse, note that $s(1, p) = p \cap \ua 1 = \{ 1 \}$ because
  $1$ is the largest element of the frame and $1 \in p$ for any filter $p$.
  Furthermore, it follows immediately from the definition of $s$ that
  $x \fleq y$ implies $s(y, p) \subseteq s(x, p)$.
  So~\eqref{it:sf-top} and~\eqref{it:sf-up} are satisfied.
  Lastly, to prove that~\eqref{it:sf-filt} holds we use that
  $(X, \ftop, \fmeet)$ is distributive.
  Suppose $z \in s(x \fmeet y, p) = p \cap \ua(x \fmeet y)$.
  Then $z \in p$ and $x \fmeet y \fleq z$, so that distributivity
  gives two worlds $u, v$ above $x$ and $y$ respectively such that $u \fmeet v = z$.
  Clearly $u, v \in p$, so that $u \in s(x, p)$ and $v \in s(y, p)$, as desired.
\end{proof}

\begin{theorem}\label{thm:compl-CL-refl-mp-veq}
  The logic $\log{CL}(\{ \axref{ax:refl}, \axref{ax:mp}, \axref{ax:veq} \})$
  is sound and complete with respect to the class of L-frames
  satisfying~\eqref{eq:int-corr}.
\end{theorem}
\begin{proof}
  Soundness follows from Lemma~\ref{lem:sf-int-corr}.
  For completeness suppose $\phi \cp \psi$ is not derivable in
  $\log{CL}(\{ \axref{ax:refl}, \axref{ax:mp}, \axref{ax:veq} \})$.
  Then there exists a conditional L-space $\topo{X} = (X, \ftop, \fmeet, s, \tau)$
  satisfying $s(x, a) = a \cap \ua x$ for all $x \in X$ and clopen filters $a$.
  As a consequence of the fact that every Heyting algebra is distributive~\cite[Proposition 1.5.3]{esakia},
  the dual conditional lattice $\Fclp(\topo{X})$ is
  distributive, and hence $(X, \ftop, \fmeet)$ is a distributive semilattice.
  Lemma~\ref{lem:sf-distr} then informs us that setting $\hat{s}(x, p) = p \cap \ua x$
  yields a conditional L-frame $(X, \ftop, \fmeet, \hat{s})$.
  Since $\hat{s}$ extends $s$, the valuation for $\topo{X}$ that invalidates
  $\phi \cp \psi$ does the same for $(X, \ftop, \fmeet, \hat{s})$.
  This proves completeness.
\end{proof}

  We note that in this semantics we have:
  \begin{align*}
    x \Vdash \phi \cto \psi
      &\iff s(x, \llb \phi \rrb) \subseteq \llb \psi \rrb \\
      &\iff \llb \phi \rrb \cap \ua x \subseteq \llb \psi \rrb \\
      &\iff \forall y \fgeq x ( y \Vdash \phi \text{ implies } y \Vdash \psi)
  \end{align*}
  which should remind the reader of the usual interpretation of
  implication in intuitionistic Kripke frames.
  Furthermore, in light of Lemma~\ref{lem:sf-distr} we can do away with
  the selection function altogether.

  Thus, we can interpret intuitionistic logic in distributive semilattices
  with a valuation that assigns to each proposition letter a filter of the
  semilattice. Let $(X, \ftop, \fmeet, V)$ be such a distributive semilattice
  with a valuation, then the interpretation of an intuitionistic formula $\phi$
  is defined recursively by:
  \begin{equation*}
    x \Vdash \phi \to \psi
      \iff \forall y \fgeq x(y \Vdash \phi \text{ implies } y \Vdash \psi)
  \end{equation*}
  together with the conditions from Definition~\ref{def:L-model}.
  We call this the \emph{semilattice semantics} of intuitionistic logic.
  We obtain the following theorem as a consequence of
  Theorem~\ref{thm:compl-CL-refl-mp-veq}.

\begin{theorem}
  Intuitionistic logic is sound and complete with respect to
  semilattice semantics.
\end{theorem}

\section{Future work}\label{sec:conc}

  We studied non-distributive positive logic with a weak notion of implication.
  Guided by the semantics for (classical) conditional logic, we provided a
  semantic framework for the logic given by semilattices with a selection
  function. By augmenting these frames with a topology, a duality for
  (the algebraic semantics of) the logic was derived, which we then used to
  prove completeness.
  
  We started with a weak form of implication so that one can modularly add
  additional axioms. We showcased this flexibility by giving sound and complete
  semantics for the extension of $\log{CL}$ with (combinations of) a number of
  axioms. This revealed an interesting phenomenon, namely the need to extend
  a topological selection function (whose second argument only takes
  clopen filters) to a selection function that can handle all filters.
  We also discovered that there seems to be no canonical fill-in to handle all axioms,
  as different axioms sometimes require different fill-ins in their completeness proofs.

  The work in this paper opens up many avenues for future work, some of
  which we briefly discuss below.
  
  \begin{description}
    \item[Extensions with more complex axioms]
          The axioms considered in Section~\ref{sec:extensions}
          are all relatively simple in the sense that they contain at
          most one occurrence of $\cto$. In future work, we intend to
          investigate other extensions, such as
          transitivity and cautious monotonicity:
          \begin{enumerate}\itemsep=0pt
            \renewcommand{\labelenumi}{(\theenumi) }
            \renewcommand{\theenumi}{$\mathsf{tr}$}
            \item $(p \cto q) \wedge (q \cto r) \cp (p \cto r)$
            \renewcommand{\theenumi}{$\mathsf{cm}$}
            \item $(p \cto q) \wedge (p \cto r) \cp ((p \wedge q) \cto r)$.
          \end{enumerate}
    \item[Sahlqvist correspondence and canonicity]
          Sahlqvist correspondence for classical modal logic
          identifies a large class of axioms for which one can mechanically
          find a frame correspondent. Sahlqvist completeness
          subsequently proves completeness for the extension of modal logic
          with a set of Sahqlvist formulas with respect to the class of Kripke
          frames satisfying the corresponding frame conditions,
          see for example \cite[Sections~3.6 and~5.6]{BlaRijVen01}.
          Similar results have been proven for a wide range of logics,
          including both weak positive logic and its modal extension investigated
          in~\cite{BezEA24}.
          
          It would be interesting to see to what extend we can obtain
          similar results for $\log{CL}$. The multitude of fill-ins for topological
          selection functions may make this more challenging,
          as each fill-in may require its own formula shape.
          Related work in the intuitionistic setting was recently
          presented in~\cite{DufGro24}.
    \item[Filtrations]
          Filtrations provide a method of turning a model that falsifies
          some formula or consequence pair into a finite model, such that
          it still falsifies the given formula or consequence pair.
          Therefore they can be used to obtain finite model properties.
          While filtrations for classical conditional logics are well
          understood~\cite{Seg89,Nut88}, it appears that the non-distributivity
          of $\log{CL}$ frustrates attempts to carry over the same definitions
          in our setting. Further research is required to resolve this,
          and therewith obtain a finite model property.
    \item[Diamond-like conditional implication]
          Throughout this paper we focussed on a conditional implication
          whose second argument behaves like a ``normal box,'' in the sense that it
          preserves finite meets. But one could wonder what happens if
          we add another implication, $\diamondcto$ whose second argument
          behaves like a diamond. This is particularly interesting because
          in the modal extension of weak positive logic studied in~\cite{BezEA24},
          boxes and diamonds relate in an unexpected way: whereas boxes are normal
          and diamonds are only monotone, and their interaction is described by
          one of the two interaction axioms of (distributive)
          positive modal logic~\cite{Dun95}.
  \end{description}

{\footnotesize
\bibliographystyle{ieeetr}
\bibliography{bibliography.bib}
}


\clearpage
\appendix

\section{Correspondence results}

\subsection{Correspondence results for Section~\ref{subsec:top-fill-in}}
\label{subsec:top-fill-in-proof}

  The next four lemmas prove the four correspondence
  results from Proposition~\ref{prop:triv-fill-corr}.

\begin{lemma}
  A general frame $\mo{G} = (X, \ftop, \fmeet, s, A)$ validates
  \eqref{ax:refl} $\top \cp p \cto p$ if and only if it satisfies
  \eqref{eq:refl-corr}: $s(x, a) \subseteq a$ for all $x \in X$ and $a \in A$.
\end{lemma}
\begin{proof}
  Suppose $\mo{G}$ satisfies \eqref{eq:refl-corr} and let $V$ be
  an admissible valuation for $\mo{G}$.
  Then $V(p) \in A$ and by assumption $s(x, V(p)) \subseteq V(p)$ for
  all $x \in X$. This implies $x \Vdash p \cto p$ for all $x$,
  hence the $(\mo{G}, V) \Vdash \top \cp p \cto p$.
  Since $V$ is arbitrary, we find that $\mo{G}$ validates
  $\top \cp p \cto p$.
  
  For the converse, suppose $\mo{G}$ does not satisfy \eqref{eq:refl-corr}.
  Then there exist $x \in X$ and $a \in A$ such that $s(x, a) \not\subseteq a$.
  Let $V$ be a valuation of the proposition letters such that $V(p) = a$.
  Then $(\mo{G}, V), x \not\Vdash p \cto p$. This implies that $\mo{G}$ does not
  validate $\top \cp p \cto p$.
\end{proof}

\begin{lemma}
  A general frame $\mo{G} = (X, \ftop, \fmeet, s, A)$ validates
  \eqref{ax:cond} $p \cp (\top \cto p)$ if and only if it satisfies
  \eqref{eq:cond-corr}: $x \in a$ implies $s(x, X) \subseteq a$
  for all $x \in X$ and $a \in A$.
\end{lemma}
\begin{proof}
  Suppose $\mo{G}$ satisfies \eqref{eq:cond-corr} and let $V$ be an admissible valuation for 
  $\mo{G}$. Suppose $(\mo{G}, V), x \Vdash p$, then $x \in V(p)$ so by assumption $s(x,X)
  \subseteq V(p)$ thus $(\mo{G}, V), x \Vdash \top \cto p$. Since $x$ and $V$ are arbitrary this
  proves that $\mo{G}$ validates $p \cp (\top \cto p)$.

  For the converse, suppose $\mo{G}$ does not satisfy \eqref{eq:cond-corr}.
  Then there exist $x \in X$ and $a \in A$ such that $x \in a$ and $s(x, X) \not \subseteq a$. Let $V$ be a valuation
  such that $V(p) = a$. Then $(\mo{G}, V), x \Vdash p$, but $(\mo{G}, V), x \not \Vdash \top \cto
  p$ so $\mo{G}$ does not validate $p \cp \top \cto p$.
\end{proof}

\begin{lemma}
  A general frame $\mo{G} = (X, \ftop, \fmeet, s, A)$ validates
  \eqref{ax:veq} $p \cp q \cto p$ if and only if it satisfies
  \eqref{eq:veq-corr}: $x \in a$ implies $s(x, b) \subseteq a$
  for all $x \in X$ and $a, b \in A$.
\end{lemma}
\begin{proof}
  Suppose $\mo{G}$ satisfies \eqref{eq:veq-corr} and let $V$ be an admissible valuation for
  $\mo{G}$ and suppose $(\mo{G}, V), x \Vdash p$. Then $x \in V(p)$. so by assumption $s(x,V(q))
  \subseteq V(p)$, hence $(\mo{G}, V), x \Vdash p \cto q$. Since $x$ and $V$ are arbitrary this
  proves that $\mo{G}$ validates $p \cp q \cto p$.

  For the converse, suppose $\mo{G}$ does not satisfy \eqref{eq:veq-corr}. Then there exist $x \in X$ and $a, b \in A$ 
  such that $x \in a$ and $s(x,b) \not\subseteq a$. Let $V$ be a valuation of the proposition letters
  such that $V(p) = a$ and $V(q) = b$. Then $(\mo{G}, V), x \Vdash p$ but $(\mo{G}, V), x \not
  \Vdash q \cto p$ and so $\mo{G}$ does not validate $p \cp q \cto p$.
\end{proof}

\begin{lemma}
  A general frame $\mo{G} = (X, \ftop, \fmeet, s, A)$ validates
  \eqref{ax:cs} $p \wedge q \cp p \cto q$ if and only if it satisfies
  \eqref{eq:cs-corr}: $x \in a \cap b$ implies $s(x, a) \subseteq b$
  for all $x \in X$ and $a, b \in A$.
\end{lemma}
\begin{proof}
  Suppose $\mo{G}$ satisfies \eqref{eq:cs-corr} and let $V$ be
  an admissible valuation for $\mo{G}$.
  Suppose $(\mo{G}, V), x \Vdash p \wedge q$.
  Then $x \in V(p) \cap V(q)$ so by assumption $s(x, V(p)) \subseteq V(q)$,
  hence $(\mo{G}, V), x \Vdash p \cto q$.
  Since $x$ and $V$ are arbitrary this proves that $\mo{G}$ validates
  $p \wedge q \cp p \cto q$.
  
  For the converse, suppose $\mo{G}$ does not satisfy \eqref{eq:refl-corr}.
  Then there exist $x \in X$ and $a, b \in A$ such that
  $x \in a \cap b$ while $s(x, a) \not\subseteq b$.
  Let $V$ be a valuation of the proposition letters such that $V(p) = a$
  and $V(q) = b$. Then $(\mo{G}, V), x \Vdash p \wedge q$ but $x \not\Vdash p \cto q$,
  so $\mo{G}$ does not validate $p \wedge q \cp p \cto q$.
\end{proof}

  We now prove the ``moreover'' part of Proposition~\ref{prop:triv-fill-corr}
  for general frames that satisfy the HMS separation axiom.
  This suffices because both full and descriptive general frames satisfy it.

\begin{lemma}
  Let $\mo{G} = (X, \ftop, \fmeet, s, A)$ be a general frame that satisfies
  the HMS separation axiom. Then $\mo{G}$ validates
  \eqref{ax:veq} $p \cp q \cto p$ if and only if it satisfies
  \eqref{eq:veq-corr-fd}: $s(x, a) \subseteq \ua x$
  for all $x \in X$ and $a \in A$.
\end{lemma}
\begin{proof}
  The direction from right to left follows from the fact that any general
  frame that satisfies \eqref{eq:veq-corr-fd} also satisfies \eqref{ax:veq}.

  For the converse, suppose $\mo{G}$ does not sastify \eqref{eq:veq-corr-fd}. 
  Then there exists $x \in X$ and $a \in A$ such that $s(x, a) \not \subseteq \ua x$.
  So there must be some $y \in s(x, a)$ such that $x \not \fleq y$.
  As a consequence of the HMS separation axiom we can find
  some $b \in A$ containing $x$ but not $y$.
  In particular this implies $s(x, a) \not\subseteq b$.
  Now let $V$ be a valuation for $\mo{G}$ such that $V(p) = b$ and $V(q) = a$.
  Then $(\mo{G}, V), x \Vdash p$ because $x \in b = V(p)$,
  but $(\mo{G}, V), x \not\Vdash q \cto p$ because $s(x, V(q)) \not\subseteq V(p)$.
\end{proof}

\begin{lemma}
  Let $\mo{G} = (X, \ftop, \fmeet, s, A)$ be a general frame that satisfies
  the HMS separation axiom. Then $\mo{G}$ validates
  \eqref{ax:cs} $p \wedge q \cp p \cto q$ if and only if it satisfies
  \eqref{eq:cs-corr-fd}: $x \in a $ implies $s(x, a) \subseteq \ua x$
  for all $x \in X$ and $a \in A$.
\end{lemma}
\begin{proof}
  The direction from right to left follows from the fact that any general
  frame that satisfies \eqref{eq:cs-corr-fd} also satisfies \eqref{eq:cs-corr}.
  
  For the converse, suppose $\mo{G}$ does not satisfy \eqref{eq:refl-corr}.
  Then there exist $x \in X$ and $a \in A$ such that
  $x \in a$ while $s(x, a) \not\subseteq \ua x$.
  So there must be some $y \in s(x, a)$ such that $x \not\fleq y$.
  As a consequence of the HMS separation axiom we can find some
  $b \in A$ containing $y$ but not $x$.
  Now let $V$ be a valuation of the proposition letters such that $V(p) = a$
  and $V(q) = b$. Then by construction we have $x \Vdash p \wedge q$ but
  $x \not\Vdash p \cto q$, so $\mo{G}$ does not validate $p \wedge q \cp p \cto q$.
\end{proof}

\subsection{Correspondence results for Section~\ref{subsec:fill-in-agn}}

\begin{lemma}\label{lem:det-corr}
  A general frame $\mo{G} = (X, \ftop, \fmeet, s, A)$ validates
  \eqref{ax:det} $(\top \cto p) \cp p$ if and only if it satisfies
  \eqref{eq:det-corr}: if $s(x, X) \subseteq a$ then $x \in a$
  for all $x \in X$ and $a \in A$.
\end{lemma}
\begin{proof}
  Suppose $\mo{G}$ satisfies \eqref{eq:det-corr}.
  Let $x \in X$ and let $V$ be an admissible valuation for $\mo{G}$.
  If $(\mo{G}, V), x \Vdash \top \cto p$ then $s(x, X) \subseteq V(p)$ and since
  $V(p) \in A$ the assumption tells us that $x \in V(p)$.
  Therefore $(\mo{G}, V), x \Vdash p$.
  $x \in s(x, X)$ by assumption we find $(\mo{G}, V), x \Vdash p$.
  Since this holds for all worlds and valuations, $\mo{G}$ validates
  $\axref{ax:det}$.
  
  For the converse, suppose that \eqref{eq:det-corr} does not hold.
  Then there exist $x \in X$ and $a \in A$ such that $s(x, X) \subseteq a$
  while $x \notin a$. Let $V$ be a valuation such that $V(p) = a$.
  Then we have $(\mo{G}, V), x \Vdash \top \cto p$ while $(\mo{G}, V), x \not\Vdash p$,
  so $\mo{G}$ does not validate $\axref{ax:det}$.
\end{proof}

\begin{lemma}
  A general frame $\mo{G} = (X, \ftop, \fmeet, s, A)$ validates
  \eqref{ax:expl} $\top \cp (\bot \cto p)$ if and only if it satisfies
  \eqref{eq:expl-corr}: $s(x, \{ \ftop \}) = \{ \ftop \}$ for all $x \in X$.
\end{lemma}
\begin{proof}
  Suppose $\mo{G}$ satisfies \eqref{eq:expl-corr}. Let $x \in X$ and $V$ be an admissible
  valuation for $\mo{G}$. By assumption $s(x,\{\ftop\}) = \{\ftop\} \subseteq p$ and so
  $(\mo{G},V), x \Vdash \bot \cto p$. Since this holds for all worlds and valuations, $\mo{G}$
  validates \eqref{ax:expl}.

  For the converse, suppose that \eqref{eq:expl-corr} does not hold.
  Then there exist $x \in X$ such that $s(x,\{\ftop\}) \not = \{\ftop\}$. Let $V$ be any
  valuation, then $(\mo{M}, V), x \not \Vdash \bot \cto \bot$. 
  So $\mo{G}$ does not validate \eqref{ax:expl}.
\end{proof}

  For the moreover part we use either the discriteness of a full
  general frame, or the compactness of a descriptive one.

\begin{lemma}\label{lem:det-corr-fd}
  Let $\mo{G} = (X, \ftop, \fmeet, s, A)$ be a general frame that is either
  full or descriptive. Then $\mo{G}$ validates
  \eqref{ax:det} $(\top \cto p) \cp p$ if and only if it satisfies
  \eqref{eq:det-corr-fd}: $x \in s(x, X)$ for all $x \in X$.
\end{lemma}
\begin{proof}
  If $\mo{G}$ satisfies \eqref{eq:det-corr-fd} then it also satisfies
  \eqref{eq:det-corr}, so by Lemma~\ref{lem:det-corr} it validtes $\axref{ax:det}$.
  Conversely, if $\mo{G}$ does not satisfy \eqref{eq:det-corr-fd}
  then there exists a world $x \in X$ such that $x \notin s(x, X)$.
  Now we can find an admissible filter $a$ containing $s(x, X)$ but not $x$:
  \begin{itemize}
    \item if $\mo{G}$ is full we can take $a = s(x, X)$;
    \item if $\mo{G}$ is descriptive then we can use compactness to find
          such an $a$.
  \end{itemize}
  Let $V$ be a valuation such that $V(p) = a$. Then
  $(\mo{G}, V) \Vdash \top \cto p$ but $(\mo{G}, V) \not\Vdash p$,
  so $\axref{ax:det}$ is not valid on $\mo{G}$.
\end{proof}

\subsection{Correspondence results for Section~\ref{subsec:non-top}}
\label{subsec:non-top-proof}

  The next lemmas prove Proposition~\ref{prop:refl-mp-veq-corr}.

\begin{lemma}
  A general frame $\mo{G} = (X, \ftop, \fmeet, s, A)$ validates
  \eqref{ax:pnp} $p \wedge (p \cto \bot) \cp \bot$ if and only if it satisfies
  \eqref{eq:pnp-corr}: if $x \in a$ and $s(x, a) = \{ \ftop \}$ then $x = \ftop$
  for all $x \in X$ and $a \in A$.
\end{lemma}
\begin{proof}
  Suppose $\mo{G}$ satisfies \eqref{eq:pnp-corr}.
  Let $x \in X$ and let $V$ be a valuation for $\mo{G}$.
  If $(\mo{G}, V), x \Vdash p \wedge (p \cto \bot)$ then
  $x \in V(p)$ and $s(x, V(p)) = \{ \ftop \}$.
  So by assumption $x = 1$, hence $(\mo{G}, V), x \Vdash \bot$.
  
  If $\mo{G}$ does not satisfy \eqref{eq:pnp-corr},
  then there exist $x \in X$ and $a \in A$ such that
  $x \in a$ and $s(x, a) = \{ \ftop \}$ but $x \neq \ftop$.
  Now if $V$ is a valuation such that $V(p) = a$, we have
  $(\mo{G}, V), x \Vdash p \wedge (p \cto \bot)$ while
  $(\mo{G}, V), x \not\Vdash \bot$.
\end{proof}

\begin{lemma}\label{lem:mp-corr}
  A general frame $\mo{G} = (X, \ftop, \fmeet, s, A)$ validates
  \eqref{ax:mp} $p \wedge (p \cto q) \cp q$ if and only if it satisfies
  \eqref{eq:mp-corr}: $x \in a$ and $s(x,a) \subseteq b$ implies $x \in b$
  for all $x \in X$ and $a,b \in A$.
\end{lemma}
\begin{proof}
  Suppose $\mo{G}$ satisfies \eqref{eq:mp-corr}. Let $x \in X$ and let $V$ be a valuation for
  $\mo{G}$. If $(\mo{G}, V), x \Vdash p \wedge (p \cto q)$, then $x \in V(p)$ and $s(x,V(p)) 
  \subseteq V(q)$ so $x \in V(q)$ hence $(\mo{G}, V), x \Vdash q$. 
  Since $x$ and $V$ were arbitrary, $\mo{G}$ validates \eqref{ax:mp}.

  For the converse, suppose that \eqref{eq:mp-corr} does not hold. Then there exist $x \in X$
  and $a,b \in A$ such that $x \in a$ and $s(x,a) \subseteq b$, but $x$ is not in $b$.
  Let $V$ be a valuation such that
  $V(p) = a$ and $V(q) = b$, then $(\mo{G}, V), x \Vdash p \land (p \cto q)$, but $(\mo{G},
  V), x \not \Vdash q$. So $\mo{G}$ does not validate \eqref{ax:mp}.
\end{proof}

  We now prove the ``moreover'' part of Proposition~\ref{prop:refl-mp-veq-corr}.

\begin{lemma}
  Let $\mo{G} = (X, \ftop, \fmeet, s, A)$ be a general frame that is either
  full or descriptive. Then $\mo{G}$ validates
  \eqref{ax:mp} $p \wedge (p \cto q) \cp q$ if and only if it satisfies
  \eqref{eq:mp-corr-fd}: $x \in a$ implies $x \in s(x, a)$
  for all $x \in X$ and $a \in A$.
\end{lemma}
\begin{proof}
  If $\mo{G}$ satisfies \eqref{eq:mp-corr-fd} then it also satisfies
  \eqref{eq:mp-corr}, so by Lemma~\ref{lem:mp-corr} it validtes $\axref{ax:mp}$.

  Conversely, suppose there exist $x \in X$ and $a \in A$ such that
  $x \in a$ but $x \notin s(x, a)$. Then in the same way as in
  Lemma~\ref{lem:det-corr-fd} we can find a filter $b$ containing
  $s(x, a)$ which does not contain $x$.
  Let $V$ be a valuation such that $V(p) = a$ and $V(q) = b$.
  Then $(\mo{G}, V), x \Vdash p \wedge (p \cto q)$ but
  $(\mo{G}, V), x \not\Vdash q$, so $\axref{ax:mp}$ is not valid.
\end{proof}

\end{document}